\documentclass[12pt,letterpaper]{amsart}
\usepackage{amsmath,mathrsfs,amsthm}
\usepackage{txfonts,rotating,tikz}
\usetikzlibrary{intersections,shapes.arrows,calc}

\usepackage[all,cmtip]{xy}
\usepackage{amssymb}
\usepackage{amsfonts}
\usepackage{amscd}
\usepackage{color} 
\usepackage[utf8]{inputenc}
\usepackage{hyperref}
\usepackage{fullpage}
\usepackage{epic}
\usepackage{mathtools}
\usepackage{tikz}
\usepackage{tikz-cd}
\usetikzlibrary{chains}
\usepackage[thinlines]{easytable}
\usepackage{array}
\usepackage{ctable}
\usepackage{pbox}
\usepackage[usestackEOL]{stackengine}
\usepackage[sort, numbers]{natbib}
\tikzcdset{scale cd/.style={every label/.append style={scale=#1},
    cells={nodes={scale=#1}}}}
\DeclareMathOperator{\Res}{\mathrm{Res}}

\newcolumntype{?}{!{\vrule width 1pt}}

\newtheorem{thm}{Theorem}[section]

\newtheorem*{theorem*}{Theorem}

\newtheorem{theorem}{Theorem}[section]

\newtheorem{lemma}[theorem]{Lemma}
\newtheorem{proposition}[theorem]{Proposition}
\newtheorem{corollary}[theorem]{Corollary}

\newtheorem{remark}[theorem]{Remark}

\theoremstyle{definition}
\newtheorem{definition}[theorem]{Definition}
\newtheorem{notation}[theorem]{Notation}
\newtheorem{example}[theorem]{Example}

\newcommand{\be}{\begin{equation}}
\newcommand{\ee}{\end{equation}}
\newcommand{\et}{\end{theorem}}
\newcommand{\bd}{\begin{definition}}
\newcommand{\ed}{\end{definition}}
\newcommand{\bp}{\begin{proposition}}
\newcommand{\ep}{\end{proposition}}
\newcommand{\bl}{\begin{lemma}}
\newcommand{\el}{\end{lemma}}
\newcommand{\bco}{\begin{corollary}}
\newcommand{\eco}{\end{corollary}}
\newcommand{\br}{\begin{remark}}
\newcommand{\er}{\end{remark}}
\newcommand{\bex}{\begin{example}}
\newcommand{\eex}{\end{example}}
\newcommand{\ben}{\begin{enumerate}}
\newcommand{\een}{\end{enumerate}}
\newcommand{\bc}{\begin{cases}}
\newcommand{\ec}{\end{cases}}

\newcommand{\bpf}{\begin{proof}}
\newcommand{\epf}{\end{proof}}

\newcommand{\bma}{\begin{bmatrix}}
\newcommand{\ema}{\end{bmatrix}}

\newcommand{\barK}{\bar{ \mathcal{K} } }
\newcommand{\barO}{\bar{ \mathcal{O} } }

\newcommand{\scrG}{\mathscr{G} }
\newcommand{\scrT}{\mathscr{T} }
\newcommand{\calG}{\mathcal{G} }

\newcommand{\fh}{\mathfrak{h}}

\newcommand{\fb}{\mathfrak{b}}
\newcommand{\fg}{\mathfrak{g}}

\newcommand{\cal}{\mathcal}

\DeclareMathOperator{\Gr}{\mathtt{Gr}}

\newcommand{\beqn}{\begin{equation}}
\newcommand{\eeqn}{\end{equation}}

\usepackage{setspace,tikz,xcolor,mathrsfs,listings,multicol,amssymb}
\usepackage{rotating}
\usepackage[vcentermath]{youngtab}
\usepackage{enumerate}
\usepackage{booktabs}
\usepackage[centertableaux]{ytableau}
\usepackage[all,cmtip]{xy}
\usetikzlibrary{arrows,matrix}
\tikzset{tab/.style={matrix of math nodes,column sep=-.35, row sep=-.35,text height=7pt,text width=7pt,align=center,inner sep=2,font=\footnotesize}}

\newcommand{\arxiv}[1]{\href{https://arxiv.org/abs/#1}{\texttt{arXiv:#1}}}

\newcommand{\g}{\mathfrak{g}}
\newcommand{\h}{\mathfrak{h}}

\newcommand{\fw}{\omega} 
\newcommand{\coroot}{\check{\alpha}}  
\newcommand{\inner}[2]{\left\langle #1, #2 \right\rangle}
\newcommand{\iso}{\cong}

\newcommand{\abs}[1]{\lvert #1 \rvert}

\newcommand{\field}{\mathbf{k}}

\newcommand{\te}{\widetilde{e}}
\newcommand{\tf}{\widetilde{f}}

\DeclareMathOperator{\wt}{wt} 
\DeclareMathOperator{\ad}{ad} 
\DeclareMathOperator{\Span}{span} 

\newcommand{\mcB}{\mathcal{B}}

\newcommand{\ba}{\mathbf{a}}

\newcommand{\VV}{\mathbb{V}}

\newcommand{\ZZ}{\mathbb{Z}}

\newcommand{\CC}{\mathbb{C}}

\definecolor{darkred}{rgb}{0.7,0,0} 
\newcommand{\defn}[1]{{\color{darkred}\emph{#1}}} 

\definecolor{UQgold}{RGB}{196, 158, 54} 
\definecolor{UQpurple}{RGB}{73, 7, 94} 
\definecolor{UMNgold}{RGB}{255,200,46} 
\definecolor{UMNmaroon}{RGB}{106,0,50} 

\usepackage{listings}
\lstdefinelanguage{Sage}[]{Python}
{morekeywords={False,True},sensitive=true}
\definecolor{dblackcolor}{rgb}{0.0,0.0,0.0}
\definecolor{dbluecolor}{rgb}{0.01,0.02,0.7}
\definecolor{dgreencolor}{rgb}{0.2,0.4,0.0}
\definecolor{dgraycolor}{rgb}{0.30,0.3,0.30}

\theoremstyle{plain}

\newtheorem{prop}[thm]{Proposition}

\numberwithin{equation}{section}

\usepackage[colorinlistoftodos]{todonotes}

\begin{document}
\title{ Smooth Locus of Twisted Affine Schubert Varieties and twisted affine Demazure modules}
\author{Marc Besson}
\address[M.\,Besson]{Yau Mathematical Sciences Center, Tsinghua University, Haidian District, Beijing, 100084, China}
\email{bessonm@tsinghua.edu.cn}
\author{Jiuzu Hong}
\address[J.\,Hong]{Department of Mathematics, University of North Carolina at Chapel Hill, Chapel Hill, NC 27599-3250, U.S.A.}
\email{jiuzu@email.unc.edu}

\maketitle
\begin{abstract}
Let $\scrG$ be a special parahoric group scheme of twisted type over the ring of formal power series over $\mathbb{C}$, excluding the absolutely special case of $A^{(2)}_{2\ell}$.  Using the methods and results of Zhu,  we prove a duality theorem for general $\scrG$:  there is a duality between the level one twisted affine Demazure modules and the function rings of  certain torus fixed point subschemes in affine Schubert varieties for  $\scrG$. Along the way, we also establish the duality theorem for $E_6$.   As a consequence,  we determine the smooth locus of any affine Schubert variety in the affine Grassmannian of $\scrG$. In particular, this confirms a conjecture of Haines and Richarz.  
\end{abstract}

\tableofcontents

\section{Introduction}

Let $G$ be an almost simple algebraic group over $\mathbb{C}$ and let $\Gr_G$ be the affine Grassmannian of $G$.  The geometry of the affine Grassmannian is related to integral highest weight representations of Kac-Moody algebras via the affine Borel-Weil theorem. Similarly, the geometry of affine Schubert varieties are closely related to  affine Demazure modules.  

 Let $T$ be a maximal torus in $G$ and let  $X_*(T)^+$  be the set of dominant coweights. For any $\lambda\in X_*(T)^+$, 
let $\overline{\Gr}_G^\lambda$ be the associated affine Schubert variety in $\Gr_G$, which is the closure of the $G(\mathcal{O})$-orbit ${\Gr}_G^\lambda$, where $\mathcal{O}=\mathbb{C}[[t]]$.   
Evens-Mirkovi\'c \cite{EM} and Malkin-Ostrik-Vybornov \cite{MOV} proved that the smooth locus of $\overline{\Gr}_G^\lambda$ is exactly the open Schubert cell $\Gr_G^\lambda$. Zhu \cite{Zh1} proved that there is a duality between the affine Demazure modules and the coordinate ring of the $T$-fixed point subschemes of affine Schubert varieties when $G$ is of type $A$ and $D$, and in many cases of the exceptional types $E_6, E_7$ and $E_8$.  As a consequence, this gives another approach to determine the smooth locus of $\overline{\Gr}_G^\lambda$ for type $A,D$ and many cases of type $E$.

In this paper,  we study a connection between the geometry of twisted affine Schubert varieties and twisted affine Demazure modules.  Following the method of  Zhu in \cite{Zh1}, we will use the weight multiplicities of twisted affine Demazure modules to determine the smooth locus of twisted affine Schubert varieties.

Let $G$ be an almost simple algebraic group of simply-laced or adjoint type with the action of a ``standard'' automorphism $\sigma$ of order $m$,  defined in Section \ref{subsect_auto}.  When $G$ is not of type $A_{2\ell}$, $\sigma$ is just a diagram automorphism. 
Assume that $\sigma$ acts on $\mathcal{O}$ by rotation of order $m$.  Let $\mathscr{G}$ be the $\sigma$-fixed point subgroup scheme of the Weil restriction group ${\rm Res}_{\mathcal{O}/  \bar{\mathcal{O}} }(G_\mathcal{O})$, where $\bar{\mathcal{O}} =\mathbb{C}[[t^m]]$.  
 Then $\mathscr{G}$ is a special parahoric group scheme over $\bar{\mathcal{O}}$ in the sense of Bruhat-Tits.  One may define the affine Grassmannian $\Gr_\scrG$ of $\scrG$. 
 Following \cite{PR, Zh2}, we will call it a twisted affine Grassmannian. For any $\bar{\lambda}$ the image of a dominant coweight $\lambda$ in the set $X_*(T)_\sigma$ of $\sigma$-coinvariants of $X_*(T)$, 
  the twisted affine Grassmannian $\Gr_\scrG$ and twisted affine Schubert varieties $\overline{\Gr}_\scrG^{\bar{\lambda}} $ share many similar properties with the usual affine Grassmannian $\Gr_G$ and affine Schubert varieties.  For instance,  a version of the geometric Satake isomorphism for $\mathscr{G}$ was proved by Zhu in \cite{Zh3}.  
  
In the literature, special parahoric group schemes are parametrized by special vertices on local Dynkin diagrams. In this paper, our approach is more Kac-Moody theoretic. For this reason, we use the terminology of affine Dynkin diagrams instead of local Dynkin diagrams.
Following \cite{HR},  there are two special parahoric group schemes for $A_{2\ell}^{(2)}$, and in this case the parahoric group scheme $\mathscr{G}$ that we consider is special but not absolutely special. We prove  Theorem \ref{thm_loc} in Section \ref{smooth_locus}, which asserts that
\begin{theorem}
\label{thm1}
For any special parahoric group scheme $\scrG$ induced from a standard automorphism $\sigma$,  the following restriction is an isomorphism:
\[H^0(\overline{\Gr}_{\mathscr{G}}^{\bar{\lambda}}, \mathscr{L}) \rightarrow H^0((\overline{\Gr}_{\mathscr{G}}^{\bar{\lambda}})^{T^{\sigma}}, \mathscr{L}|_{(\overline{\Gr}_{\mathscr{G}}^{\bar{\lambda}})^{T^{\sigma}}})   ,  \]
where  $\mathscr{L}$ is the level one line bundle on $\Gr_\scrG$,  $T^\sigma$ is the $\sigma$-fixed point subgroup of a $\sigma$-stable maximal torus $T$ in $G$ and $(\overline{\Gr}_{\mathscr{G}}^{\bar{\lambda}})^{T^{\sigma}}$ is the $T^\sigma$-fixed point subsheme of $\overline{\Gr}_{\mathscr{G}}^{\bar{\lambda}}$. 
 
\end{theorem}
The above theorem can not be extended to the absolutely special parahoric group scheme of type $A_{2\ell}^{(2)}$, as there is no level one line bundle on $\Gr_{\mathscr{G}}$ (cf.\,\cite{Zh2}).  This theorem extends Zhu's duality to the twisted setting.  The dual $H^0(\overline{\Gr}_{\mathscr{G}}^{\bar{\lambda}}, \mathscr{L})^\vee$ is a twisted affine Demazure module, see Theorem \ref{def_dem}. Hence, Theorem \ref{thm1} is a duality between twisted affine Demazure modules and the coordinate rings of the $T^\sigma$-fixed point subschemes of twisted affine Schubert varieties.  One of the motivations of the work of Zhu \cite{Zh1} is to give a geometric realization of Frenkel-Kac vertex operator construction for untwisted simply-laced affine Lie algebras. The analogue of Frenkel-Kac construction for twisted affine Lie algebras also exists in literature, see \cite{BT,FLM}.  In fact, our Theorem \ref{thm1} implies a geometric Frenkel-Kac isomorphism, see Theorem \ref{thm_FK}.

 As a conseqence of Theorem \ref{thm1},  we obtain Theorem \ref{thm_locus} and Theorem \ref{thm_A_2n}, which asserts that
\begin{theorem}
\label{intro_smooth_locus}
\mbox{}
\begin{enumerate}
\item If $\mathscr{G}$ is not of type $A_{2\ell}^{(2)}$, then
for any $\bar{\lambda}\in X_*(T)_\sigma$,  the smooth locus of the twisted affine Schubert variety $\overline{\Gr}_\scrG^{\bar{\lambda}}$ is exactly the open cell $\Gr_\scrG^{\bar{\lambda}}$. 
\item If $\mathscr{G}$ is special but not absolutely special of type $A_{2\ell}^{(2)}$, then for any $\bar{\lambda}\in X_*(T)_\sigma$,  the smooth locus of the twisted affine Schubert variety $\overline{\Gr}_\scrG^{\bar{\lambda}}$ is  the union of $\Gr_\scrG^{\bar{\lambda}}$ and possibly some other cells $\Gr_\scrG^{\bar{\mu}}$, which are completely determined in Theorem \ref{thm_A_2n}.
\end{enumerate}
\end{theorem}
When $\mathscr{G}$ is absolutely special  of type $A^{(2)}_{2\ell}$, our method is not applicable, as there is no level one line bundle on the affine Grassmannian of $\mathscr{G}$. Nevertheless,  Richarz already proved in his Diploma \cite{Ri2} that in this case the smooth locus of any twisted affine Schubert variety is the open cell,  see Remark \ref{absolutely_special_rmk}. Thus, our Theorem \ref{intro_smooth_locus} confirms a conjecture of Haines-Richarz \cite[Conjecture 5.4]{HR}. Beyond that, we also completely determine the smooth locus of twisted affine Schubert varieties for special but not absolutely special parahoric group scheme $\mathscr{G}$ of type $A_{2\ell}^{(2)}$. 
Richarz studied the twisted affine Schubert varieties in \cite{Ri2}, and determined their smooth loci in the case of absolutely special group schemes of type $A_{2\ell}^{(2)}$ and  the special parahoric group scheme of type $A_{2\ell-1}^{(2)}$. It is also worthwhile to mention that the smooth locus of the quasi-minuscule Schubert variety for $D^{(3)}_{4}$ is determined by Haines-Richarz in \cite{HR} by rather lengthy computations.  In fact, one can define special parahoric group schemes over any base field $\mathrm{k}$ of any characteristic, and the twisted Schubert variety $\overline{\Gr}_\mathscr{G}^{\bar{\lambda}}$ over the field $\mathrm{k}$. By the works \cite{HLR,HR,Lo},   Theorem \ref{intro_smooth_locus} remains true for normal twisted Schubert varieties over any field $\mathrm{k}$ (Schubert varieties are always normal if the characteristic is not bad). 

To prove Theorem \ref{thm1}, one  ingredient is Theorem \ref{thm_fixed_point2} in Section \ref{smooth_locus},  which asserts that the $T^\sigma$-fixed point ind-subscheme $(\Gr_{\scrG})^{T^\sigma} $ is isomorphic to the affine Grassmannian $\Gr_{\mathscr{T}}$, where $\mathscr{T}$ is the $\sigma$-fixed point subscheme of the Weil restriction group ${\rm Res}_{\mathcal{O}/  \bar{\mathcal{O}} }(T_\mathcal{O})$.

Let $\pi: \mathbb{P}^1\to \bar{\mathbb{P}}^1$ be the map given by $t\mapsto t^m$, where $ \bar{\mathbb{P}}^1$ is a copy of $\mathbb{P}^1$. 
Another main ingredient of the proof of  Theorem \ref{thm1} is the construction of the level one line bundle $\mathcal{L}$ on  the moduli stack ${\rm Bun}_{ \mathcal{G}} $ of $\mathcal{G}$-torsors, where $ \mathcal{G}$ is the  parahoric Bruhat-Tits group scheme obtained as the $\sigma$-fixed subgroup scheme of the Weil restriction group ${\rm Res}_{ \mathbb{P}^1 /\bar{\mathbb{P}}^1 }(G_{\mathbb{P}^1 }) $ with $G$ being simply-connected.  This is achieved in Section \ref{sect_line_bundle}.  
  It is known that the level one line bundle on ${\rm Bun}_\mathcal{G}$ does not necessarily exist for an arbitary parahoric Bruhat-Tits group scheme $\mathcal{G}$ over a smooth projective curve,  for example when $\mathcal{G}$ is of type $A_{2\ell}$, cf.\,\cite[Remark 19 (4)]{He} \cite[Proposition 4.1]{Zh2}.  In Theorem \ref{thm_descent},  when $\sigma$ is standard, we prove that there exists a level one line bundle $\mathcal{L}$ on the moduli stack ${\rm Bun}_{ \mathcal{G}} $ of $\mathcal{G}$-torsors.  Following the method of Sorger in \cite{So}, we use the non-vanishing of twisted conformal blocks to construct this line bundle on ${\rm Bun}_{\mathcal{G}}$, where the general theory of twisted conformal blocks was recently developed by Hong-Kumar in \cite{HK}.

By the work of Zhu in \cite{Zh2}, for each dominant coweight $\lambda$, one can construct a global Schubert variety $\overline{\Gr}_\mathcal{G}^\lambda$, which is flat over $\mathbb{P}^1$.  The fiber over the origin is the twisted affine Schubert variety $\overline{\Gr}_\mathscr{G}^{\bar{\lambda}}$, and the fiber over a generic point is isomorphic to the usual affine Schubert variety $\overline{\Gr}_{G}^{\lambda}$.  With the level one line bundle on ${\rm Bun}_{ \mathcal{G}} $ when $\mathcal{G}$ is simply-connected,  we can construct the level one line bundle on the global affine Schubert variety $\overline{\Gr}_\mathcal{G}^\lambda$ for $\mathcal{G}$ being either simply-connected or adjoint.  
The main idea of this paper is that,  our duality theorem for twisted affine Schubert varieties can follow from 
Zhu's duality theorem for usual affine Schubert varieties via the level one line bundle on the global affine Schubert variety $\overline{\Gr}_G^{\lambda}$.   

The proof of Theorem \ref{thm1} relies on the duality theorem of Zhu in the untwisted case.  However, Zhu only established the duality in the case of type $A, D$ and some cases of type $E_6,E_7.E_8$. To fully establish  Theorem \ref{thm1}, we need to prove the duality theorem for $E_6$ in the untwisted setting.  In the case of $E_6$, the duality has been established by Zhu when $\lambda$ is the fundamental coweight $\check{\omega}_1,\check{\omega}_2, \check{\omega}_3, \check{\omega}_5, \check{\omega}_6$ (Bourbaki labelling), and Zhu also showed that the duality theorem will hold in general if the duality also hold for $\check{\omega}_4$, which is the most difficult case. In Section \ref{duality_E_6_sect}, we establish the duality theorem for $\check{\omega}_4$. This completes the duality theorem for $E_6$ in general.  One of the main techniques is a version of Levi reduction lemma (due to Zhu) in Lemma \ref{Levi_red_lem}. In addition, we  crucially use the Heisenberg algebra action on the basic representation of affine Lie algebra, and the Weyl group representations in weight zero spaces.  To make Levi reduction lemma work for the $\omega_2$-weight space of the irreducible representation $V(\omega_4)$, we use the idea of ``numbers game'' by Proctor \cite{Pro} and Mozes \cite{Mo} which was originally used to study minuscule representations. Finally, another key step is Proposition \ref{Travis_prop}, which is verified by Travis Scrimshaw using \textsc{SageMath} \cite{Sag21}, see Appendix \ref{appendix_sect}.  

We should also mention another application of the duality theorem for simply-laced simple algebraic groups. In \cite{KTWWY}, the duality theorem is crucially used for the proof of  Hikita conjecture for the transversal slices of affine Grassmannians.  

Since the appearance of our work as a preprint  	\arxiv{2010.11357},  Pappas-Zhou gave a different proof of Haines-Richarz conjecture for absolutely special parahoric subgroups in \cite{PZ}.

 \vspace{0.5em}

\noindent {\bf Acknowledgments}:  
We would like to thank the hospitality of Max Planck institute for mathematics at Bonn during our visits in November and December of 2019, where part of the work was done.  We also would like to thank Thomas Haines, Timo Richarz, Michael Strayer and Xinwen Zhu  for  helpful conversations and  valuable comments. 
J.\,Hong is partially supported by the Simons collaboration Grant 524406, and NSF grant DMS-2001365.

\section{Main definitions}
 Let $G$ be an  almost simple algebraic group over $\mathbb{C}$ of adjoint or simply-connected type. 
 We choose a maximal torus and Borel subgroup $T \subset B \subset G$.  We denote by $X^*(T)$ the lattice of weights of $T$, and by $X_*(T)$ the lattice of coweights. Their natural pairing is denoted by $\langle, \rangle$. Let $\Phi$ denote the set of  roots of $G$, and denote by $\Phi^+$ the set of positive roots of $G$ with respect to $B$.  Let $\check{\Phi}$ denote the set of coroots, so $(\Phi, X^*(T), \check\Phi, X_*(T))$ is a root datum for $G$, and write $W$ for the Weyl group of $G$.  Let $Q$ denote the root lattice of $G$, and $\check{Q}$ the coroot lattice.
  
We follow the Bourbaki labelling of the vertices of the Dynkin diagram in \cite{Bo}.   We denote by $\{ \alpha_i \,|\, i\in I\}  $ (respectively $\{  \check{\alpha}_i \,|\, i\in I\} $ the set of simple roots in  $\Phi$ (respectively  coroots in $\check{\Phi}$), where $I$ is the set of vertices of the associated Dynkin diagram of $G$.  Let $\{ {\omega}_i \,|\,  i\in I  \}$ be the set of fundamental weights of $G$, and let $\{ \check{\omega}_i \,|\,  i\in I  \}$ be the set of fundamental coweights of $G$.   We also choose a pinning $\{ x_{\alpha_i}, y_{\alpha_i}\,|\, i\in I \} $ of $G$ with respect to  $B$ and $T$.
 
 Let $\fg, \fb,\fh$ denote the Lie algebras of $G, B, T$ respectively.  Let $\{e_i, f_i  \,|\,  i\in I \}$ denote the set of Chevalley generators associated to the pinning $\{ x_{\alpha_i}, y_{\alpha_i}\,|\, i\in I \} $. Let $e_\theta$ (resp. $f_\theta$) be the highest (resp. lowest ) root vector in $\fg$, such that $[e_\theta, f_\theta]$ is the coroot $\theta^\vee$ of $\theta$.
 \subsection{ Standard automorphisms}
 \label{subsect_auto}
Let $\sigma$ be an automorphism  of order $m$ on $G$ preserving $B$ and $T$.  Let $\tau$ be a diagram automorphism preserving $B, T$ and a pinning $\{ x_{\alpha_i}, y_{\alpha_i}\,|\, i\in I \} $. Let $r$ be the order of $\tau$.  

When $\fg$ is not $A_{2\ell}$, we take $\sigma$ to be $\tau$.  When $\fg$ is $A_{2\ell}$,  by \cite[Theorem 8.6]{Ka} there exists a unique automorphism $\sigma$ of order $m=4$ such that

\begin{equation}
\label{aut_def1}
\begin{cases}
\sigma(e_i)=e_{\tau(i)},   \quad  \text{ if } i\not= \ell, \ell+1;    \\
 \sigma(e_i)= \mathrm{i} e_{\tau(i)},   \quad \text{ if } i\in \{\ell, \ell+1\} ; \\
  \sigma(f_\theta)=f_\theta, 
\end{cases}
\end{equation}
where $\mathrm{i}$ is a square root of $-1$.   One can check that 
\begin{equation}
\label{aut_def2}
\begin{cases}
\sigma(f_i)=  f_{\tau(i)},   \quad  \text{ if } i\not= \ell, \ell+1;    \\
 \sigma(f_i)= -\mathrm{i}  f_{\tau(i)},   \quad \text{ if } i\in \{\ell, \ell+1\} ; \\
  \sigma(e_\theta)=e_\theta
\end{cases}.
\end{equation}
In fact,  $\sigma= \tau \circ   {\mathrm i}^{ h} $, where $h\in \fh$ such that
\[   \alpha_i(h)=\begin{cases}   0 ,  \quad  \text{ if }  i\not= \ell, \ell+1   \\
1, \quad   \text{ if } i=\ell, \ell+1    
      \end{cases}  .  \]
This automorphism induces a unique automorphism on $G$. We still call it $\sigma$. 
      
   We call these automorphisms on $G$ or $\fg$ ``standard'', as the fixed point Lie subalgebra $\fg^\sigma$ is the standard finite part of the associated twisted affine Lie algebra $\hat{L}(\fg,\sigma)$ (cf.\,Section \ref{Sect_BWB}) in the sense of Kac \cite[\S6.3]{Ka}. From $\sigma$, we will construct a twisted affine Grassmannian and a line bundle of level one on it. There will be no level one line bundle on the twisted affine Grassmannian associated to $\tau$ on $G$ of type $A_{2\ell}$. 
    Throughout this paper, we will only consider standard automorphisms.

The following table describe the fixed point Lie algebras for all standard  automorphisms:
 \begin{equation}
\label{Fix_table}
 \begin{tabular}{|c  | c | c |c |c |c | c| c| c|c |c|c|c|c|c |c ||} 
 \hline
$(\fg, m)$   & $(A_{2\ell-1}, 2 ) $  &  $(A_{2\ell}, 4)$  &  $(D_{\ell+1} , 2  ) $  &           $ (D_4,   3)$  &  $ (E_6,  2)$    \\ [0.8ex] 
 \hline
$ \fg^\sigma $  &  $ C_\ell $  &   $ C_\ell $  &  $B_\ell $   &    $G_2 $ &  $F_4 $  \\ [0.8ex] 
 \hline
\end{tabular},
\end{equation}
where by convention $C_1$ is $A_1$ and $\ell\geq 3$ for $D_{\ell+1}$.  
When $(\fg, m)\not=(A_{2\ell}, 4) $,  the fixed point Lie algebra $\fg^\sigma$ is well-known as listed in the above table. 
When $(\fg, m)=(A_{2\ell}, 4) $,  the fixed Lie algebra $\fg^\sigma$ is of type $C_\ell$, which can follow from the twisted Kac-Moody theory, cf.\,\cite[\S 6.3, \S 8.4]{Ka}.

Recall that we follow the Bourbaki labelling of the vertices of the Dynkin diagram. 
Set 
 \be
 \label{label_roots}
  \begin{cases}
\beta_i= {\alpha_i} |_ {\fh^\sigma}, \text{ for } i=1,2,\cdots, \ell,  \quad  \text{ if } (\fg, m)=(A_{2\ell-1}, 2), \text{ or } (D_{\ell+1}, 2)  \\
\beta_1=\alpha_1|_{\fh^\sigma}, \beta_2=\alpha_2|_{\fh^\sigma},  \quad \text{ if } (\fg, m)=(D_4, 3)\\
\beta_1=\alpha_2|_{\fh^\sigma}, \beta_2=\alpha_4|_{\fh^\sigma},\beta_3=\alpha_3|_{\fh^\sigma},\beta_4=\alpha_1|_{\fh^\sigma}, \quad  \text{ if }  (\fg, m)=(E_6, 2)\\
\beta_i= \alpha_i |_ {\fh^\sigma},   \text{ for }  i=1,2,\cdots, \ell-1; \beta_\ell= (\alpha_\ell+\alpha_{\ell+1}) |_ {\fh^\sigma}=2\alpha_\ell|_{\fh^\sigma},  \, \text{if} (\fg, m)=(A_{2\ell}, 4).
 \end{cases}  
 \ee
 
 Let  $I_\sigma$ be the set of all subscript indices of $\beta_i$.    Then for each case, the set $\{\,\beta_j  \,|\,  j\in I_\sigma \, \}$ gives rise to the set of simple roots of $\fg^\sigma$. One can see easily that this labelling will coincide with Bourbaki labelling for non simply-laced types Dynkin diagrams.

 We now define a map 
 $\eta: I\to I_\sigma$.  
 When $(\fg, m)\not=  (A_{2\ell}, 4)$,  $\eta$ is defined such that  $\beta_{\eta(i) }= \alpha_i|_{\fh^\sigma}$ for any $i\in I$. When $(\fg, m)=  (A_{2\ell}, 4)$, set
 \[ \eta(i)=\eta(2\ell+1-i)=i, \text{ for any } 1\leq i\leq  \ell.   \]
 Let $\{ \,\check{\beta}_j  \,|\,  j\in I_\sigma \}$ be the set of simple coroots of $\fg^\sigma$.  We can describe $\check{\beta}_j $ as follows:
 \be
 \label{label_roots2}
  \check{\beta}_j=   \sum_{i\in \eta^{-1}(j)  }  \check{\alpha}_{ i }.  
  \ee
The description of $\check{\beta}_j $ also appears in \cite[Section 3]{Ha} in a slightly different setting.

Let $\{\, \lambda_j \,|\, j\in I_\sigma    \,\}$ be the set of fundamental weights of $\fg^\sigma$, and let $\{\, \check{\lambda}_j \,|\, j\in I_\sigma    \,\}$ be the set of fundamental coweights of $\fg^\sigma$.  The fundamental weights can be described as follows:
\be
\label{fund_weight}
   \lambda_j=  \omega_i|_{\fh^\sigma}, \quad   \text{ for  some $i$ with } \eta(i)=j    .   \ee
In the case of fundamental coweights, we need to describe them separately.  
When $(\fg, m)\not= (A_{2\ell}, 4)$, 
\be \label{fund_weightI}
 \check{\lambda}_j=\sum_{i\in \eta^{-1}(j)} \check{\omega}_i.
\ee         
         When $(\fg,m)=(A_{2\ell}, 4)$, 
         we have 
                  \be
         \label{fund_coweightII}
         \check{\lambda}_j= \begin{cases}
        \check{\omega}_j+\check{\omega}_{2\ell+1-j},      \quad   \quad   j=1,2,\cdots,  \ell-1   \\
         \frac{1}{2}(\check{\omega}_\ell +  \check{\omega}_{\ell+1} ), \quad   \quad  j=\ell
         \end{cases} . 
         \ee

\subsection{Affine Grassmannian of special parahoric group schemes}
\label{sect2.4}
Let $\mathcal{K}$ denote the field of formal Laurent series in $t$ with coefficients in $\mathbb{C}$.  Let $\mathcal{O} $ denote the ring of formal power series in $t$ with coefficients in $\mathbb{C}$.  By abuse of notation, we still use $\sigma$ to denote the automorphism of order $m$ on $\mathcal{K}$ and $\mathcal{O}$ such that $\sigma$ acts on $\mathbb{C}$ trivially, and $\sigma(t)=\epsilon^{-1} t $, where  $\epsilon=e^{\frac{2\pi \mathrm{i}}{m}}$.  Set $\barK=\mathcal{K}^\sigma$ and $\barO=\mathcal{O}^\sigma$. Then $\barK={\mathbb{C}}((\bar{t})) $ and $\barO={\mathbb{C}}[[\bar{t}]]$, where $\bar{t}=t^m$.  

Let $\mathscr{G}$ be the smooth group scheme ${\rm Res}_{\mathcal{O}/\bar{\mathcal{O}} } (G_\mathcal{O})^\sigma $ over $\barO$,
 which represents the following group functor 
\[ R\mapsto G(\mathcal{O}\otimes_{\barO} R)^\sigma ,  \quad \text{ for any } \barO-\text{algebra } R, \]
where the $G(\mathcal{O}\otimes_{\barO} R)$ denotes the group of $\sigma$-equivariant morphisms from ${\rm Spec }\,( \mathcal{O}\otimes_{\bar{\mathcal{O}}} R)$ to $G$, where $\sigma$ acts on $\mathcal{O}$ as above and acts on $G$ as a standard automorphism defined in Section \ref{subsect_auto}. Then, $\mathscr{G}$ is a special parahoric group scheme in the sense of Bruhat-Tits, as we choose $\sigma$ to be standard. In fact, up to isomorphism, this construction exhausts all  special parahoric subgroups in $\mathscr{G}(\mathcal{K})$ when $\mathscr{G}$ is not of type $A_{2\ell}^{(2)}$, and special but not absolutely special for $A_{2\ell}^{(2)}$ in the sense of 
 \cite[\S 5]{HR}, as in this case the special fiber of  $\mathscr{G}$ has a quotient isomorphic to ${\rm Sp}_{2\ell}$. 
 \begin{remark} When $G$ is of type $A_{2\ell}$,  the parahoric group scheme $\mathscr{G}={\rm Res}_{\mathcal{O}/\bar{\mathcal{O}} } (G_\mathcal{O})^\tau $ is absolutely special of type $A_{2\ell}^{(2)}$, where $\tau$ acts on $G$ by a nontrivial diagram automorphism and acts on $\mathcal{O}$ by $t\mapsto -t$. But we will not consider this case, except in Remark \ref{absolutely_special_rmk}.  \end{remark}

We can similarly define the smooth group scheme $\mathscr{T}:={\rm Res}_{\mathcal{O}/\bar{\mathcal{O}} } (T_\mathcal{O})^\sigma $, which has connected fibers (cf.\,\cite[Lemma 4.4.16, Lemma 4.4.8]{BrT}). Note that, for general almost simple algebraic group $G$, we can still define $\mathscr{G}$ and $\mathscr{T}$, but we need to take the neutral components of  ${\rm Res}_{\mathcal{O}/\bar{\mathcal{O}} } (G_\mathcal{O})^\sigma $ and ${\rm Res}_{\mathcal{O}/\bar{\mathcal{O}} } (T_\mathcal{O})^\sigma$ respectively.  For convenience, throughout this paper we only work with $G$ being adjoint or simply-connected.

Let $L^+ \mathscr{G}$ denote the jet group  and $L\mathscr{G}$ be the loop group  of $\mathscr{G}$ over $\mathbb{C}$, that is, for all $\mathbb{C}$-algebras $R$, we set $L^+\mathscr{G}(R)=\mathscr{G}( R[[t]])$ and $L\mathscr{G}(R)= \mathscr{G}(R((t)))$. We denote by $\Gr_{\mathscr{G}}$ the affine Grassmannian of $\mathscr{G}$, which is defined as the fppf quotient $L\mathscr{G}/L^+\mathscr{G}$. 
In particular, we have 
 \[ \Gr_\mathscr{G}({\mathbb{C}})= G(\mathcal{K})^\sigma/ G(\mathcal{O} ) ^\sigma . \]
It is known that $\Gr_{\mathscr{G}}$ is a projective ind-variety, cf.\,\cite[Theorem 1.4]{PR}.  Following \cite{PR,Zh2},   we will call it a twisted affine Grassmannian of $\mathscr{G}$. We can also attach the twisted affine Grassmannian $\Gr_{ \mathscr{T} }:=L \mathscr{T}/L^+ \mathscr{T}$ of $\mathscr{T}$.  This is a highly non-reduced ind-scheme. Moreover,
\[ \Gr_{ \mathscr{T} }({\mathbb{C}})= T(\mathcal{K})^\sigma/T(\mathcal{O})^\sigma.   \]
For any $\lambda\in X_*(T)$, we can naturally attach an element $t^\lambda\in T(\mathcal{K})$.  We now define the norm $n^\lambda\in T(\mathcal{K} )^\sigma$ of $t^\lambda$, 
\begin{equation}
\label{norm_eq}
 n^\lambda:= \prod_{i=0}^{m-1}  \sigma^i(t^\lambda)=\epsilon^{\sum_{i=1}^{m-1} i \sigma^i(\lambda)  } t^{\sum_{i=0}^{m-1} \sigma^i(\lambda)}. \end{equation}

There exists a natural bijection
\be T(\mathcal{K})^\sigma/T(\mathcal{O})^\sigma \simeq  X_*(T)_\sigma , \ee
where $X_*(T)_\sigma$ denotes the set of $\sigma$-coinvariants in $X_*(T)$. Any $\bar{\lambda}\in X_*(T)_\sigma$ corresponds to the coset $n^\lambda  T( \mathcal{O})^\sigma$, where $\lambda$ is a representative of $\bar{\lambda}$.  By Theorem \cite[Theorem 0.1]{PR}, the components of $\Gr_{\scrG}$ can be parametrized by elements in $\pi_1(G)_\sigma$, where $\pi_1(G)\simeq X_*(T)/\check{Q}$, and  $(X_*(T)/\check{Q})_\sigma$ is the the set of coinvariants of $\sigma$ in $X_*(T)/\check{Q}$.

When $G$ is of adjoint type, we  describe $(X_*(T)/\check{Q})_\sigma$ in the following table.
 
 \begin{equation}
\label{Com_table}
 \begin{tabular}{|c  | c | c |c |c |c | c| c| c|c |c|c|c|c|c |c ||} 
 \hline
$(G, m)$   & $(A_{2\ell-1}, 2 ) $  &  $(A_{2\ell}, 4)$  &  $(D_{2\ell+1} , 2  ) $  &       $(D_{2\ell} , 2  ) $  &      $ (D_4,   3)$  &  $ (E_6,  2)$    \\ [0.8ex] 
 \hline
$ X_*(T)/{\check{Q}} $  &  $ \mathbb{Z}/2\ell\mathbb{Z}\ $  &   $ \mathbb{Z}/ (2\ell+1) \mathbb{Z}$  &  $ \mathbb{Z}/4\mathbb{Z} $   &  $   \mathbb{Z}/ 2 \mathbb{Z}  \times  \mathbb{Z}/2\mathbb{Z}     $     &  $  \mathbb{Z}/2\mathbb{Z}  \times  \mathbb{Z}/2\mathbb{Z}   $ &  $ \mathbb{Z}/3\mathbb{Z} $  \\ [0.8ex] 
 \hline
$ (X_*(T)/{\check{Q}} )_\sigma$  &   $ \mathbb{Z}/2\mathbb{Z}$  &   $ 0 $  &    $ \mathbb{Z}/2\mathbb{Z} $  &  $  \mathbb{Z}/2\mathbb{Z}$ &  $ 0$ &  $0 $  \\  
 \hline
\end{tabular}. 
\end{equation}
\\ [0.3ex] 

\subsection{Twisted affine Schubert varieties}
\label{sect_twisted}
Let $e_0$ be the base point in $\Gr_\mathscr{G}({\mathbb{C}})$.  For any $\bar{\lambda}\in X_*(T) _\sigma$,  let $e_{\bar{\lambda}}$ denote the point $n^\lambda e_0 \in \Gr_{\scrG}( {\mathbb{C}})$.    The point  $e_{\bar{\lambda}}$ only depends on $\bar{\lambda}\in X_*(T)_\sigma$. Let $X_*(T)^+_\sigma$ denote the set of images of $X_*(T)^+$ in $X_*(T)_\sigma$ via the projection $X_*(T)\to X_*(T)_\sigma$.  Then, we have the following Cartan decomposition for $\Gr_{\mathscr{G}}$ (cf.\,\cite[Proposition 2.8]{Ri1}),
\be
\label{cartan_dec}
\Gr_\mathscr{G}({\mathbb{C}})=  \bigsqcup_{\bar{ \lambda} \in   X_*(T)^+_\sigma }  \Gr_\mathscr{G}^{\bar{\lambda}} , \ee
where  $\Gr_\mathscr{G}^{\bar{\lambda}}:= G(\mathcal{O})^\sigma  e_{\bar{\lambda}} $.   The Schubert variety $\overline{ \Gr }_{\mathscr{G}}^{\bar{\lambda}} $ is defined to be the reduced closure of $\Gr_\mathscr{G}^{\bar{\lambda}}$ in $\Gr_{\mathscr{G}}$. Moreover, 
\[ \dim \overline{ \Gr }_{\mathscr{G}}^{\bar{\lambda}} =2\langle \lambda, \rho   \rangle ,\]
where $\rho$ is the sum of all fundamental weights of $\fg$. It is easy to see that the dimension is independent of the choice of $\lambda$.  

For any $\bar{\lambda},\bar{\mu}\in  X_*(T)^+_\sigma$,  we write $\bar{\mu}\preceq \bar{\lambda}$ if $\Gr_\mathscr{G}^{\bar{\mu}} \subseteq  \overline{ \Gr }_{\mathscr{G}}^{\bar{\lambda}} $.  For any $i\in I$, let $\overline{\check{\alpha}}_i$ denote the image of $\check{\alpha}_i$ in $X_*(T)_\sigma$. For any $j\in I_\sigma$, set 
\be
\label{gamma_root}
 \gamma_j =\overline{\check{\alpha}}_i,    \quad \text{ if $j=\eta(i)$.}\ee
 It is clear that $\gamma_j $ is well-defined.

The following lemma follows from \cite[Corollary 2.10]{Ri1}.
\begin{lemma}
\label{lem_order}
$\bar{\mu}\preceq \bar{\lambda}$ if and only if $\bar{\lambda}-\bar{\mu}$ is a non-negative integral linear combination of $\{ \,\gamma_j  \,|\,  j\in I_\sigma \,\}$. 
 \end{lemma}
By the ramified geometric correspondence \cite[\S 1]{Zh3}, the set $X_*(T)_\sigma$ can be realized as the weight lattice of the reductive group $H:=(\check{G})^\tau$, where $\check{G}$ is the Langlands dual group of $G$ and $\tau$ is a diagram autorphism on $\check{G}$ corresponding to the one on $G$, and $\{ \,\gamma_j  \,|\,  j\in I_\sigma \,\}$ is the set of simple roots for $H$.  Moreover,  $X_*(T)^+_\sigma$ is the set of dominant weights of $H$, and the partial order $\preceq$ is exactly the standard partial order for dominant weights of $H$.  

We now assume $G$ is of adjoint type.  From the perspective of the geometric Satake,  we can determine the minimal elements in $ X_*(T)^+_\sigma$, in other words the minimal Schubert variety in each connected component of $\Gr_{\scrG}$.  
From the table (\ref{Com_table}), we see that when $(G, m)=(A_{2\ell-1}, 2 )$, $\Gr_{\scrG }$ has two components, where $\Gr_{\scrG}^{  \overline{ \check{\omega} }_1  } $ is the minimal Schubert variety in the non-neutral component, since $ \overline{ \check{\omega} }_1$ gives the minuscule dominant weight of $H\simeq {\rm Sp}_{2\ell}$. When $(G, m)=(D_{\ell+1}, 2 )$,  $\Gr_{\scrG }$ also has two components and $\Gr_{\scrG}^{  \overline{ \check{\omega} }_\ell } $ is the minimal Schubert variety in the non-neutral component, since $\overline{ \check{\omega} }_\ell$ is the minuscule dominant weight of $H\simeq {\rm Spin}_{2\ell+1}$.  Otherwise,  $\Gr_{\scrG}$ has only one component.  In fact, when $(G, m)=(A_{2\ell}, 4)$, $H\simeq {\rm SO}_{2\ell+1}$, in which case the  lattice $X_*(T)_\sigma$ concides with the root lattice of $H$. 
 
Let $S$ denote the following set
\be
\label{notation_S}
 S=\begin{cases}\{ 0\}  \quad \quad  \text{ if }  (G,r)\not=   (A_{2\ell-1}, 2 ), (D_{\ell+1}, 2 ) \\
\{ 0, \check{\omega}_1 \}   \quad \quad  \text{ if } (G,r)=   (A_{2\ell-1}, 2 )   \\
\{0, \check{\omega}_\ell \}   \quad\quad  \text{ if } (G,r)=   (D_{\ell+1}, 2 )  \\
\end{cases}.
\ee
For any $\kappa\in S$, let $\Gr_{\scrG, \kappa}$ be the component of $\Gr_\scrG$ containing the Schubert variety $\Gr_\scrG^{ \bar{\kappa} }$, or equivalently containing the point $e_{\bar{\kappa}}$. Then,
\[ \Gr_\scrG= \sqcup_{\kappa\in S}   \Gr_{\scrG, \kappa} . \]

\subsection{Global affine Grassmannian of parahoric Bruhat-Tits group schemes }
\label{sect_def_global}
Let $C$ be a complex projective line $\mathbb{P}^1$ with a coordinate $t$, and with the action of $\sigma$ such that $t\mapsto \epsilon t$. Let $\bar{C}$ be the quotient curve $C/\sigma $, and let $\pi: C\to \bar{C}$ be the projection map. Then $\bar{C}$ is also isomorphc to $\mathbb{P}^1$. 
 Let $\mathcal{G}=\Res_{C/\bar{C}}(G \times C)^{\sigma}$ be the group scheme over $\bar{C}$, which is the $\sigma$-fixed point subgroup scheme of the Weil restriction $\Res_{C/\bar{C}}(G \times C)$  of the constant group scheme $G\times C$ from $C$ to $\bar{C}$. 
 Then, $\mathcal{G}$ is a parahoric Bruhat-Tits group scheme over $\bar{C}$  in the sense of Heinloth \cite[\S 1]{He}.  Let $o$ (resp.$\bar{o}$) be the origin of $C$ (resp.$\bar{C}$), and let $\infty$ (resp. $\bar{\infty}$) be the infinite point in $C$ (resp. $\bar{C}$).
 
The group scheme $\mathcal{G}$ has the following properties:
\begin{enumerate}
\item For any $y\in \bar{C}$, if $y\not=\bar{o}, \bar{\infty}$, the fiber $\mathcal{G}_{|_y}$ over $y$ is isomorphic to $G$; the restriction $\mathcal{G}_y$ to the formal disc  $\mathbb{D}_y$ around $y$ is isomorphic to the constant group scheme $G_{  \mathbb{D}_y}$ over $\mathbb{D}_y$.
\item When $y=\bar{o}$ or $\bar{\infty}$ in $\bar{C}$, $\mathcal{G}_{|_y}$  has a reductive quotient $G^\sigma$; the restriction $\mathcal{G}_y$ to $\mathbb{D}_y$ is isomorphic to the parahoric group scheme $\mathscr{G}$. 
\end{enumerate}

Similarly, we can define the parahoric Bruhat-Tits group scheme $\mathcal{T}:= \Res_{C/\bar{C}}(T \times C)^{\sigma}$.

Given an $R$-point $p \in C(R)$ we denote by $\Gamma_p\subset C_R$ the graph of $p$ where $C_R:=C\times {\rm Spec} (R)$, and denote by $\hat{\Gamma}_p$ the formal completion of $C_R$ along $\Gamma_p$, and let $\hat{\Gamma}^\times_{p} $ be the punctured formal completion along $\Gamma_p$.  Let $\bar{p}$ be the image of $p$ in $\bar{C}$. We similarly define $\bar{C}_R$, $\Gamma_{\bar{p}}$, $\hat{\Gamma}_{\bar{p}}$ and $\hat{\Gamma}_{\bar{p}}^\times $.

For any $\mathbb{C}$-algebra $R$,  we define

\begin{equation}
\label{BD_Gr}
    \Gr_{\mathcal{G}, C}(R):= \left. \left\{ \,  (p, \mathcal{P}, \beta) \, 
    \middle| \, 
    \begin{aligned}[m]
    & \, p \in C(R) \\
    &  \text{ $\mathcal{P}$ a $\mathcal{G}$-torsor on $\bar{C}$} \\
    & \, \beta: \mathcal{P}|_{\bar{C}_R \setminus  {\Gamma}_{\bar{p}}} \simeq \mathring{\mathcal{P}}|_{\bar{C}_R \setminus  \Gamma_{\bar{p}} }\\
    \end{aligned}
    \right\} \right. ,
\end{equation}
where $\mathring{\mathcal{P}}$ is the trivial $\mathcal{G}$-bundle.

The functor  $\Gr_{\mathcal{G}, C}$ is represented by an ind-scheme which is ind-proper over $C$. We call it the global affine Grassmannian $\Gr_{\mathcal{G},C}$ of $\mathcal{G}$ over $C$.  

For any $p\not=o,\infty \in C$, the fiber $\Gr_{\mathcal{G}, p}:=\Gr_{\mathcal{G}, C}|_p$ is isomorphic to the usual affine Grassmannian $\Gr_G$, and the fiber $\Gr_{\mathcal{G}, p}$ over $p=o,\infty$  is isomorphic to the twisted affine Grassmannian $\Gr_{\mathscr{G}}$ of the parahoric group scheme $\mathscr{G}$.

\begin{remark}
One can define the global affine Grassmannian $\Gr_{\mathcal{G}}$ over $\bar{C}$, see \cite[Section 3.1]{Zh2}.  The global affine Grassmannian defined above is actually the base change of $\Gr_{\mathcal{G}}$ along $\pi: C\to \bar{C}$. 
\end{remark}

We can also define the jet group scheme $L_C^+ \mathcal{G}$ over $C$  as follows,
\begin{equation}
L_C^+ \mathcal{G} (R):= \left. \left\{(p, \gamma) \, \middle| \begin{aligned} 
& \, p \in C(R) \\
& \, \text{ $\gamma$ is a trivialization of the trivial $\mathcal{G}$-torsor on $\bar{C}$ along $ \hat{\Gamma}_{\bar{p}}$} \\
\end{aligned}
\right\} \right.
\end{equation}
Again, $L_C^+ \mathcal{G}$  is the base change of the usual jet group scheme $L^+ \mathcal{G}$ of $\mathcal{G}$ along $\pi: C\to \bar{C}$.   For any $p\not=o,\infty\in C$, the fiber $L_C^+ \mathcal{G}|_p$ is isomorphic to the jet group scheme $L^+G$ of $G$, and the fiber $L_C^+ \mathcal{G}|_p$ over $p=o,\infty$ is isomorphic to jet group scheme $L^+ \mathscr{G}$.

We have a left action of $L_C^+ \mathcal{G}$ on $\Gr_{\mathcal{G}, C}$ given by 
\begin{equation}
  (  (p, \gamma), (p, \mathcal{P}, \beta) )\mapsto (p, \mathcal{P}', \beta), 
\end{equation}
where $\mathcal{P}'$ is obtained by choosing a trivialization of $\mathcal{P}$ along $\hat{\Gamma}_{\bar{p}}$ and then composing this trivialization with $\gamma$ and regluing with $\beta$. 

We also can define the global loop group $L_C\mathcal{G}$ of $\mathcal{G}$ over $C$, 

\begin{equation}
L_C \mathcal{G} (R):= \left. \left\{(p, \gamma) \,\middle| \begin{aligned} 
& \, p \in C(R) \\
& \, \text{ $\gamma$ is a trivialization of the trivial $\mathcal{G}$-torsor on $\bar{C}$ along $\hat{\Gamma}_{\bar{p}}^\times $ } \\
\end{aligned}
\right\}       \right.    .
\end{equation}

Then $\Gr_{\calG, C}$ is isomorphic to the fppf quotient $L_C\scrG/L_C^+\scrG$. We can also define 
 $L_C^+ \mathcal{T}$ and $L _C\mathcal{T}$  similarly. Then,
 \[ L_C \mathcal{T}|_p\simeq   \begin{cases}  T_{\mathcal{K}_p }   \quad  \text{ if } p\not=o ,\infty   \\
 \mathscr{T}  \quad  \text{ if } p=o, \infty
\end{cases} ,  \]
where $\mathcal{K}_p$ is the field of formal Laurant series of $C$ at $p$.

\subsection{Global Schubert varieties}
\label{sect_2.7}
For each $p\in C$, we can attach a lattice $X_*(T)_p$, 
\[ X_*(T)_p=\begin{cases} X_*(T)   \quad  \text{ if } p\not=o ,\infty   \\
X_*(T)_\sigma  \quad  \text{ if } p=o, \infty
\end{cases} .    \]
By \cite[Proposition 3.4]{Zh2},  for any $\lambda\in X_*(T)$  there exists a section $s^\lambda:  C\to  L_C\mathcal{T}$,  such that for any $p\in C$,  the image  of $s^\lambda(p)$ in $X_*(T)_p$ is given by 
\[  \begin{cases}   \lambda\in X_*(T)    \quad  \text{ if } p\not=o,\infty      \\
\bar{\lambda}\in X_*(T)_\sigma   \quad  \text{ if } p=o,\infty
   \end{cases}   . \]
This naturally gives rise to $C$-points in $\Gr_{\mathcal{T}, C}$ and $\Gr_{\mathcal{G},C}$, which will still be denoted by $s^\lambda$.  Following \cite[Definition 3.1]{Zh2}, for each $\lambda\in X_*(T)$ we define the global Schubert variety $\overline{\Gr}_{\mathcal{G},C}^\lambda$ to be the minimal $L_C^+\mathcal{G}$-stable irreducible closed subvariety of $\Gr_{\mathcal{G},C }$ that contains $s^\lambda$.  
Then,  \cite[Theorem 3]{Zh2} asserts that
\begin{theorem}
\label{thm_flat_fiber}
The global Schubert variety $\overline{\Gr}_{\mathcal{G},C}^\lambda$ is flat over $C$, and 
for any $p\in C$ the fiber  $\overline{\Gr}_{\mathcal{G},p}^\lambda$ is reduced and
\[  \overline{\Gr}_{\mathcal{G},p}^\lambda\simeq     \begin{cases} \overline{\Gr}^\lambda_G     \quad  \text{ if } p\not=o,\infty   \\
   \overline{\Gr}_{\mathscr{G}}^{\bar{\lambda}    }      \quad  \text{ if } p=o,\infty
     \end{cases} . \]
\end{theorem}

%

\section{Construction of level one line bundle on ${\rm Bun}_{\mathcal{G}}$}
\label{sect_line_bundle}

In this section, we keep the assumption that $G$ is of adjoint type with the action of a standard automorphism $\sigma$.

\subsection{Borel-Weil-Bott theorem on $\Gr_\mathscr{G}$}
\label{Sect_BWB}
Let 
$\hat{L}(\fg,\sigma):=\fg(\mathcal{K})^\sigma\oplus \mathbb{C} K$ be the twisted affine algebra as a central extension of the twisted loop algebra $\fg(\mathcal{K})^\sigma$  with the canonical center $K$, whose Lie bracket is defined as follows,
  \begin{equation} 
[x[f]+z K, x'[f'] +z' K] =
[x,x' ][ff'] +m^{-1}{\rm Res}_{t=0} \,  \bigl (({df})
f'\bigr) ( x,x') K,  
 \end{equation}
for  $x[f],x'[f']\in \fg(\mathcal{K})^\sigma$ where $x,x'\in \fg$, $f,f'\in \mathcal{K}$, and $z, z'\in\mathbb{C}$; where
${\rm Res}_{t=0}$ denotes the coefficient of $t^{-1}dt$, and $(,)$ is the normalized Killing form on $\fg$, i.e. $(\check{\theta}, \check{\theta} )=2$.

We use $P( \sigma,c)$ to denote the set of highest weights of $\fg^\sigma$ which parametrizes the integrable highest weight modules of $\hat{L}(\fg, \sigma)$  of level $c$, see \cite[Section 2]{HK}.  For each $\lambda\in P( \sigma,c)$, we  denote by $\mathscr{H}_{ c}(\lambda) $ the associated integrable highest weight module of $\hat{L}(\fg, \sigma)$.

 Recall that $\{ \lambda_i \,|\, i\in I_\sigma\}$ be the set of fundamental weights of $\fg^\sigma$, where we follow the labellings in (\ref{label_roots}). Also,  $\{ \, \check{\beta}_i  \,|\, i\in I_\sigma \, \}$ is the set of simple coroots of $\fg^\sigma$. 
\begin{lemma}
\label{lem_level1}
For a standard automorphism $\sigma$, we have 
\[  P(\sigma,1)=\begin{cases}\{ 0\}  \quad \quad  \text{ if }  (\fg,m)\not=   (A_{2\ell-1}, 2 ), (D_{\ell+1}, 2 ) \\
\{ 0, \lambda_1 \}   \quad \quad  \text{ if } (\fg,m)=   (A_{2\ell-1}, 2 )   \\
\{0,   \lambda_\ell \}   \quad\quad  \text{ if } (\fg,m)=   (D_{\ell+1}, 2 )  \\
\end{cases}.\]
\end{lemma}
\begin{proof}
We first consider the case when $(\fg,m)\not= (A_{2\ell}, 4 )$.   We can read from \cite[Lemma 2.1]{HK},  for any $\lambda\in (\fh^\sigma)^\vee$,  $\lambda\in P(\sigma,1)$ if and only if
 \[  \langle  \lambda,  \check{\beta}_i \rangle\in \mathbb{Z}_{\geq 0} \quad \text{ for any } i\in I_\sigma   ,   \]
 and $\langle \lambda,  \check{\theta}_0  \rangle\leq 1$, where $\theta_0$ is the highest short root of $\fg^\sigma$ and $\check{\theta}_0$ is the coroot of $\theta_0$, and hence $\check{\theta}_0$  is the highest coroot of $\fg^\sigma$.  In this case,  $\lambda\in P( \sigma,1)$ if and only if $\lambda=0$ or a minuscule dominant weight of $\fg^\sigma$ (cf.\,\cite[Lemma 2.13]{BH}). Following the labellings in \cite[Table Fin,p53]{Ka}, when $\fg^\sigma$ is of type $C_\ell$, $\lambda_1$ is the only minuscule weight; when $\fg^\sigma$ is of type $B_\ell$,  $\lambda_\ell$ is the only minuscule weight. Any other non simply-laced Lie algebra has no minuscule weight. This finishes the argument of the lemma when $(\fg,m)\not= (A_{2\ell}, 4 )$.
 
 Now, we assume that $(\fg,m)= (A_{2\ell}, 4 )$. In this case, it is more convenient to choose a different set of simple roots for $\fg^\sigma$, rather than the one described in (\ref{label_roots}).  Namely, we can  choose 
 \[  \{ \alpha_i|_{\fh^\sigma}  \,|\, i=1,2,\cdots, \ell-1 \}\cup \{   -\theta|_{\fh^\sigma}  \} \]
 as a set of simple roots of $\fg^\sigma$. With this set of simple root, we can also read from \cite[Lemma 2.1]{HK},  for any $\lambda\in (\fh^\sigma)^\vee$,  $\lambda\in P(\sigma,1)$ if and only if $\lambda=0$.

\end{proof}

\begin{remark}
\label{rmk_level}
It is not true that $0\in P(\sigma,1)$ for any automorphism $\sigma$. For example, $0\not\in P(\tau,1)$, when $\fg=A_{2\ell}$ and $\tau$ is a diagram automorphism; instead $0\in P(\tau,2)$.
\end{remark}

We define the following map 
\be
\label{iota_map}
 \iota:   X_*(T)  \to  (\fh^\sigma  )^\vee , \ee
such that for any ${\lambda}\in X_*(T) $, $\iota({\lambda})(h)=(\lambda, h)$, where we regard $\lambda$ as an element in $\fh$ and $(,)$ is the normalized Killing form on $\fh$.  It is clear that $\iota(0)=0$. 
This map naturally descends to a map  $X_*(T)_\sigma \to   (\fh^\sigma  )^\vee$. By abuse of notation, we still call it $\iota$.  

Recall some terminilogy introduced in Section \ref{subsect_auto}.  $I_\sigma$ is the set parametrizing simple roots of $\fg^\sigma$, and we also defined a map $\eta: I\to I_\sigma$.  
The set $\{ \check{\lambda}_j \,|\, j\in I_\sigma \}$ is the set of fundamental coweights of $\fg^\sigma$,  and  $\{ \lambda_j \,|\, j\in I_\sigma \}$ is the set of fundamental weights of $\fg^\sigma$.  We also recall that $\check{\alpha}_i$ is a simple coroot of $\fg$  for each $i\in I$, and $\gamma_j$ is the image of $\check{\alpha}_i$ in $X_*(T)_\sigma$.  The following lemma already appears in \cite[Lemma 3.2]{Ha} in a slighly different setting. 
\begin{lemma}
\label{lem_iota_coroot}
For any $j\in I_\sigma$, we have 
\[   \iota( \gamma_j ) =  \begin{cases}
\beta_j ,    \quad \text{ if }  (\fg, m)\not=  (A_{2\ell}, 4), \text{ or },   (\fg, m)=  (A_{2\ell}, 4) \text{ and } j\not= \ell \\
\frac{1}{2} \beta_\ell, \quad \text{ if } (\fg, m)=  (A_{2\ell}, 4) \text{ and } j= \ell 
\end{cases}  . \]

\end{lemma}
\begin{proof}
By the definition of $\iota$, for any $\gamma_j=\overline{\check{\alpha}}_i$ with $j=\eta(i)$,
 and $k\in I_\sigma$ we have the following equalities:
\[    \langle   \check{\lambda}_k,    \gamma_j   \rangle=   \langle   \check{\lambda}_k,   \iota( \overline{\check{\alpha}}_i  )  \rangle=( \check{\lambda}_k,  \check{\alpha}_i )=\langle\check{\lambda}_k,   \alpha_i  \rangle .\]
  Then, 
this lemma can readily follows from the description of  fundamental coweights of $\fg^\sigma $ in  (\ref{fund_weightI}) and (\ref{fund_coweightII}).
\end{proof}
 Recall the set $S$ defined in (\ref{notation_S}).
\begin{lemma}
\label{lemma_iota}
For any $i\in I$, we have $ \iota(\check{\omega}_i )  =\lambda_{\eta(i)} $. As a consequence, $\iota$ maps $X_*(T)_\sigma^+$ bijectively into the set of dominant weights of $\fg^\sigma$.
Furthermore,  $\iota$ maps $S$ bijectively into $P(\sigma,1)$.  
\end{lemma}
\begin{proof}
For any $i\in I$ and $j\in I_\sigma$, we have
\[  \langle \iota(\check{\omega }_i),   \check{\beta}_j  \rangle =(  \check{\omega}_i,  \check{\beta}_j   )=(  \check{\omega}_i, \sum_{a\in \eta^{-1}(i) } \check{\alpha}_a )=  \delta_{\eta(i) ,j } .\]
Hence, $\iota(\check{\omega }_i)=\lambda_{\eta(i)}$.

In view of Lemma \ref{lem_level1},  $\iota$ maps $S$ bijectively into $P(\sigma,1)$.
\end{proof}
\begin{remark}
\label{rmk_identification}
In view of Lemma \ref{lem_iota_coroot} and Lemma \ref{lemma_iota},  when $(G, m)\not= (A_{2\ell}, 4 )$, the root systems of $\fg^\sigma$ and $H:=(\check{G})^\tau$ can be naturally identified, where $H$ is discussed in Section \ref{sect_twisted}. Namely,  $\{ \, \overline{\check{\omega}}_i  \,|\,  i\in I  \}$ is a set of fundmental weights of $H$ corresponding to $\{ \lambda_j \,|\, j\in I_\sigma\}$ of $\fg^\sigma$, and the set of simple roots $\{ \, \gamma_j \,|\,  j\in I_\sigma\}$ corresponds to $\{\, \beta_j  \,|\, j\in I_\sigma  \,  \}$ of $\fg^\sigma$.
\end{remark}

 For any $g\in G(\mathcal{K})^\sigma$,  we can define a Lie algebra automorphism 
\be
\label{eq_adj1}
 \widehat{ {\rm Ad} }_{ g} (x[f]):= {\rm Ad}_{g}(x[f])  +  \frac{1}{m} {\rm Res}_{t=0} ( g^{-1}dg,  x[f]  ) K , \ee
 for any $x[f]\in \fg(\mathcal{K})^\sigma$, where $(,)$ is the normalized Killing form on $\fg$. By Lemma \ref{lemma_iota}, $\iota(\kappa)\in P(\sigma, 1)$  for any $\kappa\in S$. Thus, $c\iota(\kappa)\in P(\sigma, c)$ for any level $c\geq 1$. 
 
 Set 
 \be \mathscr{H}_c := \oplus_{\kappa\in S}   \mathscr{H}_c(c \iota(\kappa)).\ee
 Let $\tilde{\fg}:= \fg\otimes \mathcal{K}\oplus  \mathbb{C} K'\oplus    \mathbb{C} d' $ be the untwisted Kac-Moody algebra associated to $\fg$, where $K'$ is the canonical center and $d'$ is the scaling element.  We may define an automorphism $\sigma$ on $\tilde{\fg}$ as follows,
\[  \sigma( x[f(t)] )=\sigma(x)[f(\epsilon t )], \quad  \sigma(K')=K',  \quad  \sigma(d')=d',     \]
for any   $x[f]\in \fg\otimes \mathcal{K}$.    Then the fixed point Lie algebra $\tilde{\fg}^{\sigma}$ is exactly the twisted Kac-Moody alegbra $\tilde{L}(\fg, \sigma)$ containing $\hat{L}(\fg, \sigma)$ as the derived algebra. Following from \cite[Theorem 8.7,\S8]{Ka}, in this realization  the canonical center $K$ in $\tilde{L}(\fg, \sigma)$ is equal to $mK'$, and the scaling element $d$ in $\tilde{\fg}$ is equal to $d'$ when $\tilde{\fg}^{\sigma}$ is not $A_{2\ell}^{(2)}$, and $d=2d'$when $\tilde{\fg}^{\sigma}=A_{2\ell}^{(2)}$.

For any $g\in G(\mathcal{K})$,  one can define an automorphism $\widehat{\rm Ad}_g$ on $\tilde{\fg}$ as in \cite[Section 13.2.3]{Ku}.   From the formula $\it{loc.cit}$, it is clear that if $g\in G(\mathcal{K})^{\sigma}$, then $\widehat{\rm Ad}_g$ commutes with $\sigma$.  In particular, it follows that $\widehat{\rm Ad}_g$ restricts to an automorphism on $\tilde{L}(\fg, \sigma)$.  One may observe easily that,  restricting further to $\hat{L}(\fg, \sigma)$,  this is exactly the automorphism defined in (\ref{eq_adj1}).

By demanding that $d\cdot v_\kappa=0$ for each $\kappa\in S$, the action $\hat{L}(\fg,\sigma)$ on $\mathscr{H}$ extends uniquely to an action of $\tilde{L}(\fg,\sigma)$.

\begin{lemma}
\label{lem_int}
For any $g\in G(\mathcal{K})^\sigma$,  there exists an intertwining operator $\rho_g:  \mathscr{H}_c\simeq  \mathscr{H}_c$ such that 
\be 
\label{intw_formula}
\rho_g( x[f]\cdot v  )=  \widehat{ {\rm Ad} }_{ g} (x[f])  \cdot \rho_g(  v ),   \ee
for any $x[f]\in \fg(\mathcal{K})^\sigma$ and $v\in \mathscr{H}_c$.
In particular,  for any $\kappa\in S$,  
\be
\label{intw_formula2}
  \widehat{ {\rm Ad} }_{n^{-\kappa}}( \mathscr{H}_c(0)  )= \mathscr{H}_c(c\iota(\kappa)), \, \text{ and }  \, \widehat{ {\rm Ad} }_{n^{-\kappa}}( \mathscr{H}_c(c\iota(\kappa)) ) = \mathscr{H}_c(0)   . \ee
\end{lemma}
\begin{proof}
Let $G'$ be the simply-connected cover of $G$, and let $p: G'(\mathcal{K})^\sigma \to  G(\mathcal{K})^\sigma$ be the induced map.  Then,
\be 
 \label{dec_formula}
 G(\mathcal{K})^\sigma=\sqcup_{\kappa \in S} n^{-\kappa} \overline{G'(\mathcal{K})^\sigma   } ,  
 \ee
where $\overline{G'(\mathcal{K})^\sigma   }=p(G'(\mathcal{K})^\sigma )$.   By twisted analogue of Faltings Lemma (cf.\,\cite[ Proposition 10.2]{HK}),  for any element $g\in \overline{G'(\mathcal{K})^\sigma   }$, there exists an operator $\rho_g$  which maps $ \mathscr{H}_c(c\iota(\kappa))$ to $\mathscr{H}_c(c\iota(\kappa))$ with the desired property  (\ref{intw_formula}), for any $\kappa\in S$. By decomposition (\ref{dec_formula}),  it suffices to show that, for nonzero $\kappa$,  $n^{-\kappa}$ satisfies property (\ref{intw_formula2}).

Assume $\kappa\not=0$ in $S$. 
From the table (\ref{Com_table}), the group $(X_*(T)/{\check{Q}})_\sigma$ is at most of order $2$.  Therefore,  $n^{-2\kappa}\in \overline{G'(\mathcal{K})^\sigma   } $.  
For each $\mathscr{H}_c(c\iota(\kappa))$,  we denote the action by $\pi_{c, \kappa}: \hat{L}(\fg,\sigma)\to {\rm End}( \mathscr{H}_c(c\iota(\kappa))  ) $.  Then the property (\ref{intw_formula}) for $n^{-2\kappa}$, is equivalent to the existence of an isomorphism of representations, 
\be 
\label{eq_intw1}
\rho_{n^{-2\kappa}}:  (\mathscr{H}_c(c\iota(\kappa) ), \pi_{c,\kappa}) \simeq (\mathscr{H}_c(c\iota(\kappa) ), \pi_{c, \kappa} \circ  \widehat{ {\rm Ad} }_{n^{-2\kappa}} ) .\ee
Let $v_\kappa$ be the highest weight vector in $\mathscr{H}_c(c\iota(\kappa))$. Then $v_\kappa$ is of $\fh^\sigma$-weight $c\iota(\kappa)$. We regard $\check{\beta}_i$ as elements in $\fh^\sigma$.  By formula (\ref{eq_adj1}), 
\[ \widehat{\rm Ad}_{n^{-\kappa}}( \check{\beta}_i ) = \check{\beta}_i -  (\kappa,  \check{\beta}_i ) c= \check{\beta}_i -  \langle \iota(\kappa),  \check{\beta}_i \rangle c   .  \]

Hence, $v_\kappa$ is of $\fh^\sigma$-weight 0 and a highest weight vector in the representation 
\[   (\mathscr{H}_c(c\iota(\kappa) ), \pi_{c, \kappa} \circ  \widehat{ {\rm Ad} }_{n^{-\kappa}} ) .\]
By Schur lemma, there exists an intertwining operator $\rho_{0 \kappa}$,
\be  
\label{eq_intw2}
\rho_{0 \kappa}: (\mathscr{H}_c(0), \pi_{c, 0} )  \simeq   (\mathscr{H}_c(c\iota(\kappa) ), \pi_{c, \kappa} \circ  \widehat{ {\rm Ad} }_{n^{-\kappa}} )   . \ee
We also can regard $\rho_{0 \kappa}$ as the following intertwining operator 
\be  
\label{eq_intw3}
\rho_{0 \kappa}: (\mathscr{H}_c(0), \pi_{c, 0} \circ  \widehat{ {\rm Ad} }_{n^{-\kappa}}   )  \simeq   (\mathscr{H}_c(c\iota(\kappa) ), \pi_{c, \kappa} \circ  \widehat{ {\rm Ad} }_{n^{-2\kappa}} )    \ee

Combining  isomorphisms (\ref{eq_intw1}),(\ref{eq_intw3}), we get 
\[ (\mathscr{H}_c(c\iota(\kappa)) ,\pi_{c,\kappa})  \xrightarrow{ \rho_{n^{-2\kappa}} } (\mathscr{H}_c(c\iota(\kappa) ) , \pi_{c, \kappa} \circ  \widehat{ {\rm Ad} }_{n^{-2\kappa}} )\xrightarrow{ (\rho_{0 \kappa} )^{-1} }    (\mathscr{H}_{c}(0) , \pi_{c, 0} \circ  \widehat{ {\rm Ad} }_{n^{-\kappa}} ). \]
We define $ \rho_{n^{-\kappa}}$ to be the following operator
 \[   \rho_{n^{-\kappa}}=(\rho_{0 \kappa},   (\rho_{0 \kappa} )^{-1}   \circ  \rho_{n^{-2\kappa}} ):  \mathscr{H}_c(0) \oplus \mathscr{H}_c(c\iota(\kappa) )\simeq  \mathscr{H}_c(0 ) \oplus \mathscr{H}_c(c\iota(\kappa)) .  \]
 The map $\rho_{n^{-\kappa}}$ satisfies  property (\ref{intw_formula}).
\end{proof}

As discussed in Section \ref{sect2.4},  the components of $\Gr_{\scrG}$ are parametrized by elements in $(X_*(T)/{\check{Q}})_\sigma$.  Moreover, $\Gr_\scrG= \sqcup_{\kappa\in S}   \Gr_{\scrG, \kappa} 
$, where $S$ is defined in (\ref{notation_S}).   

  Let $\mathscr{G}'$ be the parahoric group scheme ${\rm Res}_{\mathcal{O}/\bar{\mathcal{O} }}(G'_\mathcal{O})^\sigma$ ($G'$ is the simply-connected cover of $G$), and let $L^+\mathscr{G}'$ (resp. $L\mathscr{G}' $) denote the jet group scheme (resp. loop group scheme) of $\mathscr{G}'$. 
The group $L\mathscr{G}$ acts on $L\mathscr{G}'$ by conjugation.  Set  
\[L^+\scrG'_\kappa:={\rm Ad}_{n^{-\kappa}}( L^+\scrG').\]
Then, $L^+\scrG'_\kappa$ is a subgroup scheme of $L\scrG'$.    We have 
 \be
 \label{flag_var}
  \Gr_{\scrG, \kappa}\simeq  L\scrG'/L^+\scrG_\kappa'   . \ee

By the twisted analogue of Faltings lemma (cf.\,\cite[Proposition 10.2]{HK}), there exists a group homomorphism $ L {\mathscr{G}'}  \to    {\rm PGL} ( \mathscr{H}_1(0) ) $. 
Consider the central extension 
\begin{equation} 
\label{central_ext0}
1\to \mathbb{G}_m \to {\rm GL}( \mathscr{H}_1(0)  )    \to  {\rm PGL} ( \mathscr{H}_1(0) ) \to 1 . \end{equation}
The pull-back of (\ref{central_ext0}) to $L{ \mathscr{G}}'$ defines the following canonical central extension of $L{\mathscr{G}}'$: 
\begin{equation} 
\label{central_ext0'}
1\to \mathbb{G}_m \to \widehat{L { \mathscr{G}' } }   \to  L{\mathscr{G}'}    \to 1 . \end{equation}

It is known that $\widehat{L{\mathscr{G}'} }$ is a Kac-Moody group of twisted type (up to a scaling multiplicative group) in the sense of Kumar and Mathieu, see \cite[\S 9f]{PR}.  Let $\widehat{L^+\scrG'_\kappa }$ denote the preimage of $L^+\scrG'_\kappa$ in $\widehat{L{\mathscr{G}'} }$ via the projection map $\widehat{L{\mathscr{G}'} } \to L{\mathscr{G}'} $. As the same proof as in \cite[Lemma 2.19]{BH},   $\widehat{L^+\scrG'_\kappa } $ is a parabolic subgroup in $\widehat{L{\mathscr{G}'} }$, moreover 
\be
\label{com_hom}
 \Gr_{\scrG, \kappa} \simeq \widehat{L{\mathscr{G}'} }/  \widehat{L^+\scrG'_\kappa } , \ee
i.e. $\Gr_{\scrG, \kappa} $ is a partial flag variety of the Kac-Moody group $\widehat{L{\mathscr{G}'} }$.

\begin{proposition}
\label{prop_line}
There exists a line bundle $\mathscr{L}$ on $\Gr_\mathscr{G}$ such that $\mathscr{L}$ is of level one on each component of $\Gr_\mathscr{G}$. 
\end{proposition}
\begin{proof}
We first consider the simply-connected cover $G'$ of $G$.  
By \cite[Theorem 10.7 (1)]{HK}, there exists a canonical splitting of $\widehat{L{\mathscr{G}'} }   \to  L{\mathscr{G} '}$ in the central extension (\ref{central_ext0}) over $L^+ {\mathscr{G}'}$. We may define a line bundle $\mathscr{L}$ on $\Gr_{\mathscr{G}'}= \widehat{L\mathscr{G}'}/ \widehat{L^+ \mathscr{G}'}$ via the character $ \widehat{L^+ \mathscr{G}'}:= \mathbb{G}_m\times  L^+ \mathscr{G}'\to \mathbb{G}_m$ defined via the first projection.  In fact, as the argument  in \cite[Lemma 4.1]{LS}, this line bundle is the ample generator of ${\rm Pic}(\Gr_{\mathscr{G}'})$ of level 1. This finishes the proof of part (1).

We now consider the case when $G$ is of adjoint type.  Since the neutral component $\Gr_{\mathscr{G }, \circ }$ is isomorphic to $\Gr_{\mathscr{G}' }$, we get the level one line bundle on $\Gr_{\mathscr{G}, \circ }$ induced from the one on $\Gr_{\mathscr{G}' }$. 
For any other component $\Gr_{\mathscr{G},\kappa }$, by $(\ref{com_hom})$ we have an isomorphism $\Gr_{\mathscr{G}, \circ }\simeq \Gr_{\mathscr{G}, \kappa }$. Therefore,  this gives rise to the level one line bundle on $\Gr_{\mathscr{G}, \kappa}$.

\end{proof}

The line bundle $\mathscr{L}$ on $\Gr_{\scrG}$ naturally has a  $\widehat{L{\mathscr{G}'} }$-equivariant structure, since $\mathscr{L}$ admits a unique  $\widehat{L{\mathscr{G}'} }$-equivariant structure on each component of $\Gr_\mathscr{G}$ as a partial flag variety of $\widehat{L{\mathscr{G}'} }$.  
\begin{theorem}
\label{BWB_Thm}
As representations of $\hat{L}(\fg,\sigma)$, we have
$ H^0(\Gr_{\scrG}, \mathscr{L}^c )^\vee \simeq  \mathscr{H}_c$, where $\mathscr{L}^c$ is the $c$-th power of $\mathscr{L}$. 
\end{theorem} 
\begin{proof}
By \cite[\S 9f]{PR}, $\Gr_{\scrG}$ can be identified with the partial flag variety of the Kac-Moody group $\widehat{L{\mathscr{G}'} }$ constructed by Kumar  \cite{Ku}. Then, affine Borel-Weil-Bott theorem for Kac-Moody group (cf. \cite[Theorem 8.3.11]{Ku}),  the theorem follows. 
\end{proof}

Let $v_0$ be the highest weight vector in $\mathscr{H}_c $. For any $\bar{\lambda}\in X_*(T)_\sigma$,  we define 
\be
\label{Dem_gen}
 v_{\bar{\lambda}}:=  \rho_{n^{\lambda} }(v_0) , \ee
 where $ \rho_{n^{\lambda} }$ in defined in Lemma \ref{lem_int}. 
Then $ v_{\bar{\lambda}}$ is independent of the choice of the representative $\lambda$ in $X_*(T)$ and is well-defined up to a nonzero scalar. 
\begin{lemma}
\label{lem_gen}
The $\fh^\sigma$-weight of the vector $v_{\bar{\lambda}}$ is $-c\iota(\bar{\lambda})$. 
\end{lemma}
\begin{proof}
For any $h\in \fh^\sigma$, by Lemma \ref{lem_int}, 
\[ h\cdot v_{\bar{\lambda}}= h\cdot \rho_{n^\lambda} (v_0)= \rho_{n^\lambda}  (  \widehat{ {\rm Ad} }_{ n^{-\lambda}}  ( h ) v_0   ) . \]
By the formula (\ref{eq_adj1}), we have
\[  \widehat{ {\rm Ad} }_{ n^{-\lambda}}  ( h ) = h-\langle \lambda, h \rangle K.  \]
It follows that 
\[ h\cdot v_{\bar{\lambda}}= - \langle \lambda, h\rangle  c v_{\bar{\lambda}}= -c\iota(\lambda)(h)v_{\bar{\lambda}} .\]
This concludes the proof of the lemma.
\end{proof}

\begin{definition}
\label{def_dem}
For any dominant $\bar{\lambda}\in X_*(T)^+_\sigma$, 
we define the twisted affine Demazure module $D(c, \bar{\lambda})$ as the following $\fg[t]^\sigma$-module, 
\[ D(c, \bar{\lambda}):=U(\fg[t]^\sigma)v_{\bar{\lambda}}   . \]
\end{definition}
In view of Lemma \ref{lem_gen}, $D(c, \bar{\lambda})$ contains an irreducible representation $V({-c\iota(\lambda)})$ of $\fg^\sigma$ of lowest weight $-c\iota(\lambda)$. The following theorem follows from \cite[Theorem 8.2.2 (a)]{Ku}.
\begin{theorem}
\label{them_Dem}
As $\fg[t]^\sigma$-modules,  $H^0(\overline{\Gr}^{\bar{\lambda}}_\mathscr{G} , \mathscr{L}^c )^\vee\simeq D(c, \bar{\lambda})$. 
\end{theorem}

\subsection{Construction of level one line bundles on ${\rm Bun}_{\mathcal{G}}$}

In this subsection, we consider the parahoric Bruhat-Tits group scheme $\mathcal{G}:= {\rm Res}_{C /  \bar{C}    }(G\times C )^\Gamma $ over $\bar{C}$ as in the setting of Section \ref{sect_def_global}.

Let ${\rm Bun}_{\mathcal{G}}$ be the moduli stack of $\mathcal{G}$-torsors on $\bar{C}$. It is known that ${\rm Bun}_{\mathcal{G}}$ is a smooth Artin stack (cf.\,\cite[Theorem 1]{He}).
By \cite[Theorem 3]{He}, the Picard group $ {\rm Pic}({\rm Bun}_{\mathcal{G}} ) $ of ${\rm Bun}_{\mathcal{G}}$ is isomorphc to $\mathbb{Z}$, since the group $X^*(\mathcal{G}|_{y})$ of characters for  $\mathcal{G}|_{y}$ is trivial for any $y\in \bar{C}$.  In this subsection, we will construct the ample generator $\mathcal{L}\in {\rm Pic}({\rm Bun}_{\mathcal{G}} ) $ when $G$ is simply-connected, and we will construct a level one line bundle on every component of $\Gr_{\mathcal{G}, C }$ when $G$ is of adjoint type.

By Lemma \ref{lem_level1}, we have $0\in P(\sigma,1)$ for any standard automorphism $\sigma$.  Recall that $\mathscr{H}_1(0)$ is the basic representation of level one associated to $0\in P(\sigma,1)$.  

We now define the following space of twisted covacua of level one,
\begin{equation}
  \mathscr{V}_{C, \sigma }( 0):=   \frac{   \mathscr{H}_{1}(0)  } { \fg[  t^{-1} ]^\sigma  \cdot \mathscr{H}_{1} (0) }  ,  
\end{equation}
where $\fg[  t^{-1} ]^\sigma$ is the Lie subalgebra of $\hat{L}(\fg,\sigma)$.

\begin{lemma}
\label{non_van}
The dimension of the vector space $ \mathscr{V}_{C, \sigma }( 0)$ is 1.
\end{lemma}
\begin{proof}
Let $v_0$ be the highest weight vector in $\mathscr{H}_{1} (0)$. Then
 \[  \mathscr{H}_{1}(0)=U( (t^{-1}\fg[t^{-1}]  )^\sigma  ) \cdot v_0 =U( (t^{-1}\fg[t^{-1}]  )^\sigma  )  (t^{-1}\fg[t^{-1}]  )^\sigma v_0\oplus  \mathbb{C} v_0  ,\]
 where $U( (t^{-1}\fg[t^{-1}]  )^\sigma  ) $ denotes the universal enveloping algebra of $(t^{-1}\fg[t^{-1}]  )^\sigma$.
We can write $\fg[t^{-1}]^\sigma=\fg^\sigma\oplus  (t^{-1}\fg[t^{-1}]  )^\sigma $. Hence, 

 \begin{align*}
   \fg[t^{-1}]^\sigma\cdot   \mathscr{H}_{1}(0)   &=  \fg^\sigma\cdot U( (t^{-1}\fg[t^{-1}]  )^\sigma  )  (t^{-1}\fg[t^{-1}]  )^\sigma v_0 +  U( (t^{-1}\fg[t^{-1}]  )^\sigma  )  (t^{-1}\fg[t^{-1}]  )^\sigma v_0  \\
              &=  U( (t^{-1}\fg[t^{-1}]  )^\sigma  )  (t^{-1}\fg[t^{-1}]  )^\sigma v_0 ,
\end{align*}
where the first equality holds since $\fg^\sigma\cdot v_0=0$, and the second equality holds since $\fg^\sigma$ normalizes $(t^{-1}\fg[t^{-1}]  )^\sigma$ under the Lie bracket.
Therefore, $\dim  \mathscr{V}_{C, \sigma }( 0)=1$.
\end{proof}

Let $G'  $ be the simply-connected cover of $G$.  
Recall the Heinloth uniformization theorem for ${ \mathcal{G}'} := {\rm Res}_{C/  \bar{C}    }(G'   \times C)^\Gamma $over the affine line $\bar{C}\backslash \bar{o}$ (cf.\,\cite{He}), 
\[    {\rm Bun}_{{ \mathcal{G}'}   } \simeq      G'    [t^{-1}]^\sigma  \backslash   \Gr_{{ \mathscr{G} '} }  ,\]
where $ \Gr_{{ \mathscr{G} '} }$ denotes the affine Grassmannian of $\scrG':= {\rm Res}_{ \mathcal{O}/ \barO  } (G'_{\mathcal{O}})^\sigma$, and 
  $ G'[t^{-1}]^\sigma  \backslash   \Gr_{ { \mathscr{G}'} }$ denotes the fppf quotient. 
\begin{theorem}
\label{thm_descent}
 The line bundle $\mathscr{L}$ descends to a line bundle $\mathcal{L}$ on $ {\rm Bun}_{{ \mathcal{G}' } }$ .
\end{theorem}
\begin{proof}

Let $\mathscr{L}$ be the level one line bundle on $\Gr_{\mathscr{G}' }$ constructed from Proposition \ref{prop_line}. 
To show that the line bundle $\mathscr{L}$ can descend to ${\rm Bun}_{\mathcal{G}' }$, as in the argument in \cite{So}, it suffices to show that there is a $G'[t^{-1}]^\sigma$-linearization on $\mathscr{L}$. This is equivalent to the splitting of the central extension (\ref{central_ext0'})
over $G'[t^{-1}]^\sigma$. We use the same argument as in \cite[Proposition 3.3]{So}, since the vector space $\mathscr{V}_{C, \sigma }( 0)$ is nonvanishing by Lemma \ref{non_van}, the central extension (\ref{central_ext0'}) splits over $G'[t^{-1}]^\sigma$. 
\end{proof}

We consider the projection map ${\rm pr}:  \Gr_{\mathcal{G}', C  }\to {\rm Bun}_{\mathcal{G}'} $.  By abuse of notation, we still denote by $\mathcal{L}$ the line bundle on $ \Gr_{\mathcal{G}', C  }$ pulling-back from $\mathcal{L}$ on ${\rm Bun}_{\mathcal{G}'} $. 
\begin{corollary}
The restriction of the line bundle $\mathcal{L}$ to the fiber $\Gr_{\mathcal{G}', p }$ is the ample generator of ${\rm Pic} (\Gr_{\mathcal{G}', p })$, for any $p\in C$. 
\end{corollary}
\begin{proof}
It follows from Theorem \ref{thm_descent} and \cite[Proposition 4.1]{Zh2}.
\end{proof}

The following theorem is interesting by itself, but will not be used in this paper. 
\begin{theorem}
There is a natural isomorphism 
\[  H^0({\rm Bun}_{\mathcal{G}'},  \mathcal{L}  )\simeq   \mathscr{V}_{C, \sigma }( 0)^\vee, \]
 where $\mathscr{V}_{C, \sigma }( 0)^\vee$ denotes the dual of $\mathscr{V}_{C, \sigma }( 0)$. In particular,
  \[\dim H^0({\rm Bun}_{\mathcal{G}'},  \mathcal{L}  )=1.\]
\end{theorem}
\begin{proof}
The theorem follows from the same argument as in \cite[Theorem 12.1]{HK}.
\end{proof}

Now, we would like to construct the line bundle $\cal{L}$ of level one on $\Gr_{\calG, C}$, where $\calG={\rm Res}_{C/\bar{C} }(G_C)^\sigma$ with $G$ of adjoint type.  
\begin{theorem}
There exists a line bundle $\mathcal{L}$ on $\Gr_{ \calG, C } $ such that the restriction of $\mathcal{L}$ to the fiber $\Gr_{\mathcal{G}, p }$ is the level one line bundle on $\Gr_{\mathcal{G}, p }$, for any $p\in C$. 
\end{theorem}
\begin{proof}
Set $\mathring{C}=C\backslash \{o, \infty\}$. Then, $\Gr_{\mathcal{G}, \mathring{C}}:=\Gr_{\mathcal{G}, C}|_{\mathring{C}} \simeq  \Gr_G\times \mathring{C}$. Let $M$ be the set consisting of $0$ and miniscule coweights of $G$. Then, components of $\Gr_G$ can be parametrized by $M\simeq X_*(T)/\check{Q}$. For each $\kappa\in M$, 
let $\Gr_{G, \kappa}$ denote the component of $\Gr_G$ containing $e_{\kappa}$.  Let $\Gr_{\mathcal{G}, C, \kappa}$ be the closure of $\Gr_{G,\kappa} \times \mathring{C}$ in $\Gr_{\mathcal{G}, C}$.  We call $\Gr_{\mathcal{G}, C, \kappa}$ a $\kappa$-component of $\Gr_{\mathcal{G}, C}$.

The neutral component $\Gr_{\mathcal{G}, C, 0}$ is naturally isomorphic to $\Gr_{\mathcal{G}'}$ (cf.\,\cite[Lemma 5.16]{HY}). Thus, by Theorem \ref{thm_descent} we naturally get the level one line bundle $\mathcal{L}_0$ on the neutral component $\Gr_{\calG, C ,0}$. 
For any $\kappa\in M$, let $s^{\kappa}$ be a $C$-point in $\Gr_{\calG, C} $ as defined in Section \ref{sect_2.7}. Then, the translation by $s^\kappa$ gives rise to an isomorphism 
\[  \Gr_{\mathcal{G}, C, 0} \simeq \Gr_{\mathcal{G}, C, \kappa} .\]
Accordingly,  the line bundle $\mathcal{L}_0$ can be translated to the level one line bundle $\mathcal{L}_\kappa$ on $\Gr_{\mathcal{G}, C, \kappa}$.  Note that given any two elements $\kappa,\kappa'\in M$ such that $\bar{\kappa}=\bar{\kappa'}$ in $(X_*(T)/\check{Q})_\sigma$, $\Gr_{\mathcal{G}, C, \kappa} $ and $\Gr_{\mathcal{G}, C, \kappa'} $ share the same component $\Gr_{\mathscr{G}, \bar{\kappa}}$ of $\Gr_{\mathscr{G}}$ at $o$ and $\infty$.  Then $\mathcal{L}_\kappa$ and $\mathcal{L}_{\kappa'}$ agree on  $\Gr_{\mathscr{G}, \bar{\kappa}}$ as they have the same levels.   Thus, $\{ \mathcal{L}_\kappa\}_{\kappa\in M}$ glues to be a line bundle $\mathcal{L}$ on $\Gr_{\mathcal{G}, C}$ whose restriction to $\Gr_{\mathcal{G}, p}$ is of level one, for any $p\in C$.

\end{proof}

\section{Smooth locus of twisted affine Schubert varieties}
\label{smooth_locus}

In this section, we always assume that $\sigma$ is a standard automophism on $G$, and $G$ is of adjoint type.

\subsection{$\Gr_{\mathscr{T}}$ as a fixed-point ind-subscheme of  $\Gr_{\mathscr{G}}$}

We first recall a theorem in \cite[Theorem 1.3.4]{Zh1}. 
\begin{theorem}
\label{thm_fixed_point1}
The natural morphism $\Gr_{T} \rightarrow \Gr_{G}$ identifies $\Gr_{T}$ as the $T$-fixed point ind-subscheme $(\Gr_G)^T$ of $\Gr_G$.
\end{theorem}

The original proof of this theorem is not correct (communicated to us by Richarz and Zhu independently), also see \cite[Remark 3.5]{HR2}.  A correct proof can be found in  \cite[Proposition 3.4]{HR2}, and a similar proof was known to Zhu earlier.

It is clear that $T^\sigma$ is a subgroup scheme of $L\mathscr{T}$ and $L\scrG$. Hence there is a natural action of $T^\sigma$ on $\Gr_\scrG$. 
We now prove an analogue of  Theorem \ref{thm_fixed_point1} in the setting of special parahoric group schemes. 
\begin{theorem}
\label{thm_fixed_point2}
The natural morphism $ \Gr_{\scrT} \rightarrow \Gr_{\scrG}$ identifies $\Gr_{\scrT}$ as the $T^\sigma$-fixed point ind-subscheme $(\Gr_\scrG)^{T^\sigma}$ of $\Gr_\scrG$.
\end{theorem}
\begin{proof}
Let $L^{--}G$ be the ind-group scheme represented by the following functor, for any $\mathbb{C}$-algebra $R$, 
\[  L^{--}G(R): = \ker ({\rm ev}_{\infty}: G(R[t^{-1}]) \to  G(R) ) , \]
where ${\rm ev}_{\infty}$ is the evaluation map sending $t^{-1}$ to $0$. Let $L^{--}\mathscr{G}$ be the ind-group scheme which represents the following functor , 
for any $\mathbb{C}$-algebra $R$, 
\begin{equation} 
\label{thm_4.2_eq}
L^{--}\mathscr{G}(R): = \ker ({\rm ev}_{\infty}: G(R[t^{-1}])^\sigma \to  G(R)^\sigma ) . \end{equation}
We can similarly define $L^{--}T$ and $L^{--}\mathscr{T}$. 

By the similar argument as in \cite[Lemma 2.3.5]{Zh4} or \cite[Lemma 3.1]{HR2}, we have an open embedding 
\begin{equation}
\label{open_em_eq}
 L^{--}\mathscr{G}\hookrightarrow \Gr_\mathscr{G}   \end{equation}
given by $g\mapsto  ge_0$, where $e_0$ is the base point in $\Gr_\mathscr{G}$. Let ${I}$ be the Iwahori subgroup of $L^+\mathscr{G}$, which is the preimage of $B^\sigma$ via the evaluation map ${\rm ev}: L^+\mathscr{G}\to G^\sigma$ for a $\sigma$-stable Borel subgroup $B$ in $G$. We have the following decomposition 
\be 
\label{Iwahori_dec}
\Gr_\mathscr{G}=  \bigsqcup_{\bar{ \lambda} \in   X_*(T)_\sigma } I e_{\bar{\lambda}}.   
 \ee

For each $\bar{\lambda}\in X_*(T)_\sigma$, we choose a representative $\lambda\in X_*(T)$.  The twisted Iwahori Schubert cell 
\[  I e_{\bar{\lambda}}=n^{\lambda} {\rm Ad}_{  n^{-\lambda}  }   (I )e_0\]
is contained in  $n^{\lambda}  L^{--}\mathscr{G} e_0$.  Then by the  decomposition (\ref{Iwahori_dec}), 
  $\bigcup_{\bar{\lambda}\in X_*(T)_* } n^\lambda L^{--} \mathscr{G}e_0$ is an open covering of $\Gr_\mathscr{G} $.   We may naturally regard $\Gr_\mathscr{T}$ as an ind-subscheme of $\Gr_\mathscr{G}$. Hence, we may  regard $e_0$ as the base point in $\Gr_\mathscr{T}$.
Under this convention, 
\[  \bigcup_{\lambda\in X_*(T)_\sigma} n^\lambda L^{--} \mathscr{T}e_0 = \bigcup_{\lambda\in X_*(T)_\sigma}  L^{--} \mathscr{T} n^\lambda e_0   \]
is an open covering of $\Gr_\mathscr{T}$. Therefore, it suffices to show that for each $\bar{\lambda}\in X_*(T)_\sigma$, 
\[(n^\lambda L^{--} \mathscr{G}e_0)^{T^\sigma}\simeq  n^\lambda L^{--} \mathscr{T}e_0.  \]
Further, it suffices to show that $(L^{--} \mathscr{G})^{T^\sigma} \simeq L^{--}\mathscr{T}$, where the action of $T^\sigma$ on $L^{--} \mathscr{G}$ is by conjugation.  From the proof of  \cite[Proposition 3.4]{HR2}, one may see that $(L^{--}G)^{T^\sigma}\simeq  L^{--}T$. This actually implies that $(L^{--} \mathscr{G})^{T^\sigma} \simeq L^{--}\mathscr{T}$. 
Hence, this finishes the proof of the theorem.
\end{proof}

An immediate consequence of Theorem \ref{thm_fixed_point2} is the following corollary.
\begin{corollary}
\label{cor_fix}
The  $T^\sigma$-fixed $\mathbb{C}$-point set in $\Gr_\scrG$ is 
$\{  e_{\bar{\lambda}}  \,|\,   \lambda\in  X_*(T)_\sigma \}$.  
\end{corollary}

\subsection{A duality isomorphism for twisted Schubert varieties}

Let $\Gr_G$ be the affine Grassmannian of $G$, and let $\mathrm{L}$ be the line bundle on $\Gr_G$ that is of level one on every component of $\Gr_G$.  For any $\lambda\in X_*(T)$, let $\overline{\Gr}_G^\lambda$ denote the closure of $G(\mathcal{O})$-orbit at $L_\lambda:=t^\lambda G(\mathcal{O})\in \Gr_G$.  Let $(\overline{\Gr}^\lambda_G)^T$ denote the $T$-fixed point  subscheme of  $\overline{\Gr}_G^\lambda$. Zhu \cite[Theorem 0.2.2]{Zh1} proved that 
\begin{theorem}
\label{thm_zhu}
When $G$ is simply-laced and not of type $E$, the restriction map 
\[ H^0(\overline{\Gr}_{{G}}^{{\lambda}}, \mathrm{L}) \rightarrow H^0((\overline{\Gr}_{{G}}^{{\lambda}})^{T}, \mathrm{L}|_{(\overline{\Gr}_{ {G}}^{{\lambda}})^{T}})   \]  is an isomorphism. 
\end{theorem} 
In section \ref{duality_E_6_sect}, we will show that this theorem also holds for $E_6$.  It was proved by Evens-Mirkovi\'c \cite[Thereorem 0.1b]{EM} and Malkin-Ostrik-Vybornov \cite[Corollary B]{MOV}, that the smooth locus of  $\overline{\Gr}_G^\lambda$ is the open cell $\Gr_G^\lambda$ for any reductive group $G$.  In fact, this theorem can also be deduced from Theorem \ref{thm_zhu} in the simply- laced type.

We will prove a twisted version of Theorem \ref{thm_zhu} in full generality, and as a consequence we get the similar result of Evans-Mirkovi\'c and Malkin-Ostrik-Vybornov in twisted setting. In particular, this confirms a conjecture of Haines-Richarz \cite{HR}.

From Theorem \ref{thm_fixed_point2},  we have the identification $\Gr_{\mathscr{T}} \xrightarrow {\simeq} \Gr_{\mathscr{G}}^{T^{\sigma}}$. Let $\mathscr{I}^{\bar{\lambda}}$ denote the ideal sheaf of the $T^{\sigma}$-fixed subscheme $(\overline{\Gr}_{\mathscr{G}}^{\bar{\lambda}})^{T^\sigma}$ of $\overline{\Gr}_{\mathscr{G}}^{\bar{\lambda}}$. Then we have a short exact sequence of sheaves 

\begin{equation}
0 \rightarrow \mathscr{I}^{\bar{\lambda}} \rightarrow \mathscr{O}_{\overline{\Gr}_{\mathscr{G}}^{\bar{\lambda}}} \rightarrow \mathscr{O}_{(\overline{\Gr}_{\mathscr{G}}^{\bar{\lambda}})^{T^{\sigma}}} \rightarrow 0.
\end{equation}

Recall that $\mathscr{L}$ is the line bundle on $\Gr_\mathscr{G}$ which is of level one on every component.  Tensoring the above short exact sequence with $\mathscr{L}$ and taking the functor of global sections, we obtain the following exact sequence

\begin{equation}
\label{iso_van}
\begin{aligned}
0  \rightarrow & H^0(\overline{\Gr}_{\mathscr{G}}^{\bar{\lambda}}, \mathscr{I}^{\bar{\lambda}}\otimes \mathscr{L} ) \rightarrow H^0(\overline{\Gr}_{\mathscr{G}}^{\bar{\lambda}}, \mathscr{L}) \xrightarrow{r} H^0((\overline{\Gr}_{\mathscr{G}}^{\bar{\lambda}})^{T^{\sigma}}, \mathscr{L}|_{(\overline{\Gr}_{\mathscr{G}}^{\bar{\lambda}})^{T^{\sigma}}})     \rightarrow  \cdots,
\end{aligned}
\end{equation}
where  $r$ is the restriction map. 
\begin{theorem}
\label{thm_loc}
For any  special parahoric group scheme $\mathscr{G}$ induced from a standard automorphism $\sigma$, 
the restriction map 
\[   H^0(\overline{\Gr}_{\mathscr{G}}^{\bar{\lambda}}, \mathscr{L}) \xrightarrow{r} H^0((\overline{\Gr}_{\mathscr{G}}^{\bar{\lambda}})^{T^{\sigma}}, \mathscr{L}|_{(\overline{\Gr}_{\mathscr{G}}^{\bar{\lambda}})^{T^{\sigma}}})   \]
 is an isomorphism, where $\mathscr{L}$ is the level one line bundle on $\Gr_{\mathscr{G}}$.
\end{theorem}

This theorem will follow from the following proposition and Lemma \ref{lem_ideal_flat}. 

\begin{proposition}
\label{prop_surj}
The map $r$ is a surjection.
\end{proposition}

\begin{proof}
It is well-known that  any twisted affine Schubert varietiety $\overline{\Gr}_{\mathscr{G}}^{\bar{\lambda}}$ is a usual Schubert variety in a partial affine flag variety of Kac-Moody group. See the identification (\ref{flag_var}) and an argument for untwisted case in \cite[Proposition 2.21]{BH}. 
By \cite[Theorem 8.2.2 (d)]{Ku}, we have that for any $\bar{\lambda} \succeq \bar{\mu}$ in $X_*(T)^+_\sigma$, the following restriction map
\be
\label{dem_surj}
  H^0(\overline{\Gr}_{\mathscr{G}}^{\bar{\lambda}}, \mathscr{L}) \rightarrow H^0(\overline{\Gr}_{\mathscr{G}}^{\bar{\mu}}, \mathscr{L})  \ee
is surjective, and 
\be
\label{limit_1}
 H^0(\Gr_{\mathscr{G}}, \mathscr{L})= \varprojlim H^0(\overline{\Gr}^{\bar{\lambda}}_{\mathscr{G}}, \mathscr{L}|_{\overline{\Gr}^{\bar{\lambda}}_{\mathscr{G}}}) .\ee
We also have the following surjective map
\be  
\label{limit_2}
H^0((\overline{\Gr}_{\mathscr{G}}^{\bar{\lambda}})^{T^{\sigma}}, \mathscr{L}) \rightarrow H^0((\overline{\Gr}_{\mathscr{G}}^{\bar{\mu}})^{T^{\sigma}}, \mathscr{L})     \ee
for all $\bar{\lambda} \succeq \bar{\mu}$, since these $T^\sigma$-fixed closed subschemes are affine and the morphism $(\overline{\Gr}_{\mathscr{G}}^{\bar{\mu}})^{T^{\sigma}} \xhookrightarrow{} (\overline{\Gr}_{\mathscr{G}}^{\bar{\lambda}})^{T^{\sigma}}$ is a closed embedding. Moreover, 
\[H^0((\Gr_{\mathscr{G}})^{T^\sigma} , \mathscr{L}|_{ (\Gr_{\mathscr{G}})^{T^\sigma}  })= \varprojlim H^0((\overline{\Gr}^{\bar{\lambda}}_{\scrG})^{T^{\sigma}}, \mathscr{L}|_{(\overline{\Gr}^{\bar{\lambda}}_{\mathscr{G}})^{T^{\sigma}}}).\]

 Therefore, for any $\bar{\lambda} \in X_*(T)^+_\sigma$ we have the following surjective maps
\[ H^0({\Gr}_{\mathscr{G}}, \mathscr{L}) \rightarrow H^0(\overline{\Gr}_{\mathscr{G}}^{\bar{\lambda}}, \mathscr{L}),    \quad  H^0(({\Gr}_{\mathscr{G}})^{T^{\sigma}}, \mathscr{L}) \rightarrow H^0((\overline{\Gr}_{\mathscr{G}}^{\bar{\lambda}})^{T^{\sigma}}, \mathscr{L})    .  \]

Then to prove the map
\[H^0(\overline{\Gr}_{\mathscr{G}}^{\bar{\lambda}}, \mathscr{L}) \rightarrow  H^0((\overline{\Gr}_{\mathscr{G}}^{\bar{\lambda}})^{T^{\sigma}}, \mathscr{L}|_{(\overline{\Gr}_{\mathscr{G}}^{\bar{\lambda}})^{T^{\sigma}}}) \] 
is surjective, it is sufficient to prove that the map
 \be
 \label{surj1}
 H^0(\Gr_{\mathscr{G}}, \mathscr{L}) \rightarrow H^0(  (\Gr_{\mathscr{G}})^{T^\sigma}  , \mathscr{L}|_{ (\Gr_{\mathscr{G}})^{T^\sigma}  })   \ee
  is surjective, since  we will have the following commutative diagram, for all $\bar{\lambda}$, 

\begin{equation}
\begin{tikzcd}
H^0(\Gr_{\mathscr{G}}, \mathscr{L}) \arrow[r] \arrow[d] & H^0(  (\Gr_{\mathscr{G}})^{T^\sigma}  , \mathscr{L}|_{ (\Gr_{\mathscr{G}})^{T^\sigma}  }) \arrow[d] \\ 
H^0(\overline{\Gr}^{\bar{\lambda}}_{\mathscr{G}}, \mathscr{L}|_{\overline{\Gr}^{\bar{\lambda}}_{\mathscr{G}}}) \arrow[r,"r"] & H^0((\overline{\Gr}^{\bar{\lambda}}_{\mathscr{G}})^{T^{\sigma}}, \mathscr{L}|_{(\overline{\Gr}^{\bar{\lambda}}_{\mathscr{G}})^{T^{\sigma}}}). \\
\end{tikzcd}
\end{equation}
By Theorem \ref{thm_fixed_point2},  we have $\Gr_{\mathscr{T}}\simeq ( \Gr_{\mathscr{G}})^{T^\sigma}  $.  Therefore, 
the surjectivity of the map (\ref{surj1}) follows from the following Lemma \ref{lem_surj}.  
\end{proof}

We first make a digression on Heisenberg algebras and their representations. Recall that $\fh$ is a fixed Cartan subalgebra in $\fg$.  The subspace $\hat{\fh}^{\sigma}:=( \fh_ \mathcal{K})^\sigma \oplus  \mathbb{C} K\hookrightarrow \hat{L}(\fg, \sigma)$ is a Lie subalgebra.  In fact, $\hat{\fh}^{\sigma}$ is an extended (completed) Heisenberg algebra with center $\fh^\sigma\oplus \mathbb{C} K$.  Therefore, any integrable irreducible highest weight  representation of $\hat{\fh}^{\sigma}$ is parametrized by an element $\mu\in (\fh^\sigma)^\vee$ and the level $c$, i.e. $K$ acts by the scalar $c$ on this representation. We denote this representation by $  \pi_{\mu, c}$. By the standard construction, 
\be
\label{Heis_con}
  \pi_{\mu, c}={\rm ind}_{(\fh_ \mathcal{O})^\sigma\oplus \mathbb{C} K }^ {\hat{\fh}^{\sigma}} \mathbb{C}_{\mu, c} ,\ee
where ${\rm ind}$ is the induced representation in the sense of univeral enveloping algebras ,and $ \mathbb{C}_{\mu, c}$ is the 1-dimensional module over $(\fh_{O})^\sigma\oplus \mathbb{C} K$ where the action of $(\fh_{O})^\sigma$ factors through $\fh^\sigma$.

\begin{lemma}
\label{lem_surj}
The restriction map	$H^0(\Gr_{\mathscr{G}}, \mathscr{L}^c ) \rightarrow H^0(\Gr_{\mathscr{T}}, \mathscr{L}^c |_{ \Gr_{\mathscr{T}}   })$ is surjective. 
\end{lemma}

\begin{proof}
Proving surjectivity here is equivalent to proving injectivity for the dual modules,

\[0 \rightarrow H^0(\Gr_{\mathscr{T}}, \mathscr{L}^c|_{\Gr_{\mathscr{T}}})^\vee  \rightarrow  H^0(\Gr_{\mathscr{G}}, \mathscr{L}^c)^\vee.\] 

Note that both of these spaces are modules for the Heisenberg algebra $\hat{\fh}^{\sigma}$; the morphism is a $\hat{\fh}^{\sigma}$-morphism.  Since $\mathscr{T}$ is discrete, we naturally have the following decomposition
\[ H^0(\Gr_{\mathscr{T}}, \mathscr{L}^c|_{\Gr_{\mathscr{T}}})\simeq  \bigoplus_{\bar{\lambda} \in X_*(T)_\sigma }  \mathscr{O}_{  \Gr_{\mathscr{T}}, e_{\bar{\lambda}}  }\otimes  \mathscr{L}^c|_{ e_{\bar{\lambda} } }, \]
where  $\mathscr{O}_{  \Gr_{\mathscr{T}}, e_{\bar{\lambda}}  }$ is the structure sheaf of the component of $\Gr_{\mathscr{T}}$ containing $e_{\bar{\lambda}}$.  We also notice that the identity component of $\Gr_{\mathscr{T}}$ is naturally the formal group with Lie algebra $(\fh_\mathcal{K})^\sigma/(\fh_\mathcal{O})^\sigma$. In view of the construction (\ref{Heis_con}),   we have
 \[H^0(\Gr_{\mathscr{T}}, \mathscr{L}^c|_{\Gr_{\mathscr{T}}})^\vee= \bigoplus_{\bar{\lambda} \in X_*(T)_\sigma } \pi_{-c\iota(\bar{\lambda}) , c };\] 
 where the map $\iota: X_*(T)_{\sigma} \rightarrow  (\fh^\sigma)^\vee$ is defined in (\ref{iota_map}). Since each $\pi_{-c\iota(\bar{\lambda}) , c }$ is irreducible, and generated by a $-c\iota(\bar{\lambda})$-weight vector $w_{-c\iota(\bar{\lambda})}$, it suffices to show that the morphism 
\[ \pi_{-c\iota(\bar{\lambda}) , c }   \rightarrow H^0(\Gr_{\mathscr{G}}, \mathscr{L}^c)^\vee\]
 sends $w_{-c\iota(\bar{\lambda})}$ to a nonzero vector.

By Theorem \ref{BWB_Thm}, we may define a Pl\"ucker embedding 
\[  \phi: \Gr_{\mathscr{G}} \rightarrow \mathbb{P}( \mathscr{H}_c )\]
 given by $ge_0\mapsto [\rho_g ( v_0 )]$ for any $ge_0\in \Gr_\scrG$, where $\rho_g$ is defined in Lemma \ref{lem_int}, and $ [\rho_g ( v_0) ]$ represents the line in $\mathscr{H}_c $ that contains $ \rho_g ( v_0) $, where $v_0$ is the highest weight vector in $\mathscr{H}_c$.  Then we may pick a linear form $f_{\bar{\lambda}}$ on  $\mathscr{H}_c$ which is nonzero on the line $[v_{\bar{\lambda}}]$ containing the extremal weight vector $v_{\bar{\lambda}}$, and which is 0 on other weight vectors, where $v_{\bar{\lambda}}$ is defined in (\ref{Dem_gen}).  The restriction $f_{\bar{\lambda}}|_{\phi(\Gr_{\mathscr{G}})}$ produces a nontrivial element in $H^0(\Gr_{\mathscr{G}}, \mathscr{L})$, since $\phi(e_{\bar{\lambda}})=v_{\bar{\lambda}}$. 
 
 Observe that the map $ \pi_{-c\iota(\bar{\lambda}) , c } \rightarrow H^0(\Gr_{\mathscr{G}}, \mathscr{L}^c)^\vee$ sends $w_{-c\iota(\bar{\lambda})}$ to a nonzero scalar of $v_{\bar{\lambda} }$.  
 Thus the map $ \pi_{-c\iota(\bar{\lambda}) , c } \rightarrow H^0(\Gr_{\mathscr{G}}, \mathscr{L}^c)^\vee$ is nontrivial and thus injective.
\end{proof}

By Lemma \ref{lem_surj},  we obtain the following short exact sequence 
 \[0 \rightarrow H^0(\overline{\Gr}_{\mathscr{G}}^{\bar{\lambda}}, \mathscr{I}^{\bar{\lambda}}\otimes \mathscr{L}) \rightarrow H^0(\overline{\Gr}_{\mathscr{G}}^{\bar{\lambda}}, \mathscr{L}) \xrightarrow{r} H^0(\overline{\Gr}_{\mathscr{G}}^{\bar{\lambda}}, \mathscr{L} \otimes \mathscr{O}_{(\overline{\Gr}_{\mathscr{G}}^{\bar{\lambda}})}) \rightarrow 0.\]

Thus, the obstruction to the map $r$ being an isomorphism is the vanishing of the first term $ H^0(\overline{\Gr}_{\mathscr{G}}^{\bar{\lambda}}, \mathscr{I}^{\bar{\lambda}} \otimes \mathscr{L} )$.  

Let $\mathrm{I}^{\lambda}$ denote the ideal sheaf of the $T$-fixed subscheme on $\overline{\Gr}_G^{\lambda}$.  We will show that the vanishing of the first term can be deduced from the vanishing of $H^0(\overline{\Gr}_G^{\lambda}, \mathrm{I}^{\lambda} \otimes \mathrm{L})$.  \vspace{0.2em}

Recall that $\overline{\Gr}_{\mathcal{G}, C}^{\lambda}$ is a global Schubert variety defined in Section \ref{sect_2.7}.  The  constant group scheme $T^\sigma \times C$ over $C$ is naturally a closed subgroup scheme of $\mathcal{T}$.  Hence $T^\sigma$ acts on $\overline{\Gr}_{\mathcal{G}, C}^{\lambda}$ naturally.  
Let $(\overline{\Gr}_{\mathcal{G}, C}^{\lambda})^{T^\sigma}$ be the $T^\sigma$-fixed  subscheme of $\overline{\Gr}_{\mathcal{G}, C}^{\lambda}$, and let $\mathcal{I}^{\lambda}$ be the ideal sheaf of $(\overline{\Gr}_{\mathcal{G}, C}^{\lambda})^{T^\sigma}$.  Then, $\mathcal{I}^\lambda|_p$ is the ideal sheaf of $(\overline{\Gr}_{\mathcal{G}, C}^{\lambda}|_{p})^{T^\sigma}$.  Recall that, 
\[\overline{\Gr}_{ {\mathcal{G}},o}^{\lambda} = \overline{\Gr}_{\mathscr{G}}^{\bar{\lambda}}, \quad \overline{\Gr}_{ {\mathcal{G}},\infty}^{\lambda}\simeq  \overline{\Gr}_{\mathscr{G}}^{\bar{\lambda}},  \quad \overline{\Gr}_{\mathcal{G},p \neq o,\infty }^{\lambda} \simeq \overline{\Gr}_{G}^{\lambda}.\]
 In particular, we have 
 \[  \mathcal{I}^\lambda|_o= \mathscr{I}^\lambda, \quad  \mathcal{I}^\lambda|_\infty\simeq \mathscr{I}^\lambda,     \quad  \mathcal{I}^\lambda|_{p\not= o,\infty}\simeq \mathrm{I}^\lambda.\]

\begin{lemma}
\label{lem_ideal_flat}
The ideal $\mathcal{I}^\lambda$ is flat over $C$. 
\end{lemma}
\begin{proof}
Consider $\overline{\Gr}_{\mathcal{G}, {C} \setminus \{o,\infty\}}^{\lambda}$ and the $T^{\sigma}$-fixed subscheme $(\overline{\Gr}_{\mathcal{G}, {C} \setminus \{ o,\infty \}}^{\lambda})^{T^{\sigma}}$. We denote by $Z^{\lambda}$ the flat closure of $(\overline{\Gr}_{\mathcal{G}, C \setminus \{o,\infty \}}^{\lambda})^{T^{\sigma}}$ in $\Gr_{\mathcal{G}, C}$.  Since $Z$ is the closure of a $T^{\sigma}$-fixed subscheme, we see that $Z^{\lambda}|_o \subset \overline{\Gr}^{\lambda}_{\mathcal{G}, C}|_o$, and $Z^{\lambda}|_\infty \subset \overline{\Gr}^{\lambda}_{\mathcal{G}, C}|_\infty.$ 

To show $\mathcal{I}^\lambda$ is flat over $C$, it is sufficient to show that $(\overline{\Gr}_{\mathcal{G}, C}^{\lambda})^{T^\sigma}$ is flat over $C$.  This is equivalent to showing $Z^\lambda=(\overline{\Gr}_{\mathcal{G}, C}^{\lambda})^{T^\sigma}$.  In particular, it  suffices to show the fibers $Z^{\lambda}|_o$ and $Z^{\lambda}|_\infty$ are isomorphic to  $(\overline{\Gr}_{\mathscr{G}}^{\bar{\lambda}})^{T^{\sigma}}$.  Since the fiber $Z^{\lambda}|_\infty$ at $\infty$ is similar to the fiber $Z^{\lambda}|_o$ at $o$, it suffices to show that $Z^{\lambda}|_o=(\overline{\Gr}_{\mathscr{G}}^{\bar{\lambda}})^{T^{\sigma}}$.   
Note that both of these are finite schemes, we can compare the dimensions of their structure sheaves as follows:
\begin{align*}
    \dim \mathscr{O}_{(\overline{\Gr}_{\scrG}^{\bar{\lambda}})^{T^{\sigma}}}  \geq \dim \mathscr{O}_{Z^{\lambda}|_o}  &= \dim \mathscr{O}_{(\overline{\Gr}_{\mathcal{G},p \neq o ,\infty}^{\lambda} )^{T^\sigma}}    \\
    & = \dim \mathscr{O}_{(\overline{\Gr}_{\mathcal{G},p \neq o,\infty }^{\lambda} )^{T}}    \\
    & = \dim H^0( \overline{\Gr}_\mathcal{G}^{\lambda}, \mathcal{L}|_{p\neq o,\infty} )   \\
    & = \dim H^0(\overline{\Gr}_{\mathscr{G}}^{\bar{\lambda}}, \mathscr{L})  \\
    & \geq \dim \mathscr{O}_{(\overline{\Gr}_{\scrG}^{\bar{\lambda}})^{T^{\sigma}}},
\end{align*}
where the first equality follows from the flatness of $Z^\lambda$ over $C$, the third equality follows from Theorem \ref{thm_zhu} and Theorem \ref{thm_E_6}, the fourth equality follows since $\overline{\Gr}_{\mathcal{G},C}^{{\lambda}}$ is flat over $C$ (cf.\,Theorem \ref{thm_flat_fiber}), and the last inequality follows from Proposition \ref{prop_surj} .
From this comparision, it follows that $\dim \mathscr{O}_{Z^{\lambda}|_o}  = \dim \mathscr{O}_{(\overline{\Gr}_{\mathcal{G},p \neq o }^{\lambda} )^{T^\sigma}}$. Hence, $\mathscr{O}_{Z^{\lambda}|_o}  = \mathscr{O}_{(\overline{\Gr}_{\mathcal{G},p \neq 0 }^{\lambda} )^{T^\sigma}}$. This concludes the proof of the  lemma. 
\end{proof}

\begin{proof}[Proof of Theorem \ref{thm_loc}]
By Lemma \ref{lem_ideal_flat} and the properness of $\overline{\Gr}_{\mathcal{G}, C}^{\lambda}$ over $C$, we have 
\begin{equation}
\label{Euler_eq}
 \sum_{i\geq 0} (-1)^i \dim H^i(\overline{\Gr}_{G}^{\lambda}, \mathrm{I}^{\lambda} \otimes \mathrm{L})=\sum_{i\geq 0} (-1)^i  \dim H^i(\overline{\Gr}_\scrG^{\bar{\lambda}}, \mathscr{I}^{\bar{\lambda}} \otimes \mathscr{L})  .  \end{equation}

By \cite[9.h.]{PR} and \cite[Theorem 8.2.2]{Ku}, we have $H^i(\overline{\Gr}_{G}^{\lambda}, \mathrm{L})=0$ and $H^i(\overline{\Gr}_\scrG^{\bar{\lambda}}, \mathscr{L})=0 $ for any $i>0$. 
From \cite[Section 2.2]{Zh1} when $G$ is of type $A$ and $D$ and from Section \ref{duality_E_6_sect} when $G$ is of type $E_6$,  we have the following vanishing 
$H^i(\overline{\Gr}_{G}^{\lambda}, \mathrm{I}^{\lambda} \otimes \mathrm{L})=0$ for any $\lambda\in X_*(T)^+$, by considering the exact sequence (\ref{iso_van}) in the untwisted case.  From Lemma \ref{lem_surj} and the long exact sequence (\ref{iso_van}), we can see easily that $H^i(\overline{\Gr}_\scrG^{\bar{\lambda}}, \mathscr{I}^{\bar{\lambda}} \otimes \mathscr{L})=0$ for any $i\geq 1$. 
Hence,   the equality (\ref{Euler_eq}) implies that $H^0(\overline{\Gr}_\scrG^{\bar{\lambda}}, \mathscr{I}^{\bar{\lambda}} \otimes \mathscr{L})=0$ for any $\bar{\lambda}\in X_*(T)^+_\sigma$.
Therefore, the theorem finally follows from the long exact sequence (\ref{iso_van}).
\end{proof}


As an application of Theorem \ref{thm_loc}, we get a geometric Frenkel-Kac isomorphism for twisted affine  algebras. 
\begin{theorem}
\label{thm_FK}
For any special parahoric group scheme $\mathscr{G}$ induced from a standard automorphism $\sigma$, the restriction map	
\[ H^0(\Gr_{\mathscr{G}}, \mathscr{L} ) \rightarrow H^0(\Gr_{\mathscr{T}}, \mathscr{L}|_{ \Gr_{\mathscr{T}}   })\]
 is an isomoprhism, via the embedding $\Gr_\mathscr{T}\to \Gr_\mathscr{G}$.
\end{theorem}
\begin{proof}
By Theorem \ref{thm_fixed_point2},  it suffices to show that the restriction map $r:H^0(\Gr_{\mathscr{G}}, \mathscr{L} ) \rightarrow H^0(\Gr_{\mathscr{T}}, \mathscr{L}|_{( \Gr_{\mathscr{G}   })^{T^\sigma}    })  $ is an isomorphism.  In view of (\ref{limit_1}) and (\ref{limit_2}) and as a consequence of Theorem \ref{thm_loc}, the restriction map $r$ is  an isomorphism.

\end{proof}

\subsection{Application: Smooth locus of twisted affine Schubert varieties}
\label{subsect_6.2}
We now wish to investigate the smooth locus of the Schubert variety $\overline{\Gr}_{\scrG}^{\bar{\lambda}}$.

\begin{theorem}
\label{thm_locus}
Assume that $\mathscr{G}$ is not of type $A_{2\ell}^{(2)}$. For any $\lambda\in X_*(T)_\sigma^+$,  the smooth locus of $\overline{\Gr}_{\scrG}^{\bar{\lambda}}$ is precisely the open Schubert cell $\Gr_{\scrG}^{\bar{\lambda}}$.
\end{theorem}
\begin{proof}
For any $\bar{\mu}\in X_*(T)^+_\sigma$, if $e_{\bar{\mu}}=n^\mu e_0$ is a smooth point in $\overline{\Gr}_{\scrG}^{\bar{\lambda}}$, then by \cite[Lemma 2.3.3]{Zh1} $\dim \mathscr{O}_{ (\overline{\Gr}_{\scrG}^{\bar{\lambda}})^{T^\sigma} , e_{\bar{\mu}} } =1$. 

By Theorem \ref{them_Dem}, we have $ H^0(\overline{\Gr}^{\bar{\lambda}}_\mathscr{G}, \mathscr{L} )^\vee\simeq D(1, \bar{\lambda})$,
where $D(1, \bar{\lambda}) $ is the Demazure module defined in Definition \ref{def_dem}.
Then by Theorem \ref{thm_loc}, we have 
 \[  \dim D(1,\bar{\lambda})_{- \iota(\bar{\mu}) }= {\rm length} \big ( \mathscr{O}_{ (\overline{\Gr}_{\scrG}^{\bar{\lambda}})^{T^\sigma} , e_{\bar{\mu}} }  \big ),   \]
 where $D(1,\bar{\lambda})_{-\iota(\bar{\mu} ) } $ is the $-\iota(\bar{\mu})$-weight space in $D(1,\bar{\lambda})$.  
 We will prove that for any $\bar{\mu}\precneqq \bar{\lambda}$,  $\dim D(1,\bar{\lambda})_{- \iota(\bar{\mu}) }\geq 2$, which would imply that $e_{\bar{\mu}}$ is not a smooth point in $\overline{\Gr}_{\scrG}^{\bar{\lambda}}$.
 From the surjectivity of (\ref{dem_surj}), we have an embedding $D(1, \bar{\mu})\hookrightarrow  D(1,\bar{\lambda})$.  On the the other hand, $V(-\iota(\bar{\lambda}))\hookrightarrow D(1, \bar{\lambda})$, where $V(-\iota(\bar{\lambda}))$ is the irreducible representation of $\fg^\sigma$ of lowest weight $-\iota(\bar{\lambda})$.  In view of Lemma \ref{lem_order},  Lemma \ref{lem_iota_coroot} and Lemma \ref{lemma_iota}, when $G$ is not of type $A_{2\ell}$,  the relation $\bar{\mu}\precneqq \bar{\lambda}$ implies that $\iota(\bar{\mu})\precneqq \iota(\bar{\lambda})$. Hence,  $V(-\iota(\bar{\lambda}))_{ -\iota(\bar{\mu}) }\not=0$.
Furthermore,  as subspaces in $D(1, \bar{\lambda})$, 
  \[  D(1, \bar{\mu})\cap  V(-\iota(\bar{\lambda}))=0.\]
   It follows that $\dim D(1,\bar{\lambda})_{- \iota(\bar{\mu}) }\geq 2$.  This concludes the proof of the theorem. 
\end{proof}


Now we will deal with the case of $A_{2\ell}^{(2)}$. Recall the group $H=(\check{G})^\tau$ mentioned in Section \ref{sect_twisted}. By the ramified geometric Satake, $(X_*(T)_\sigma,  X_*(T)^+_\sigma, \gamma_j, j\in I_\sigma  )$ can be regarded as the weight lattice, the set of dominant weights, and simple roots of $H$. 
When $(G,m)=(A_{2\ell}, 4)$, $H$ is  $B_\ell$ of adjoint type.  Let $\varpi_1,\varpi_2,\cdots, \varpi_\ell$ be the set of fundamental dominant weights of $H$.
\begin{theorem}
\label{thm_A_2n}
 Let $\mathscr{G}$ be of type $A_{2\ell}^{(2)}$. For any $\lambda\in X_*(T)_\sigma^+$,  the smooth locus of $\overline{\Gr}_{\scrG}^{\bar{\lambda}}$ is exactly the union of $\Gr_{\scrG}^{\bar{\lambda}}$ and those $\Gr_{\scrG}^{\bar{\mu}}$ such that $ \bar{\lambda}-\bar{\mu}=\sum_{j=i}^\ell  \gamma_j$ and $\mu=\sum_{k=1}^{i-1 } a_k\varpi_k$ with all $a_k\in \mathbb{Z}^{\geq 0}$, for some $1\leq i\leq \ell$.
\end{theorem}
\begin{proof}

We first prove the following result:  for any $\bar{\lambda}, \bar{\mu}\in X_*(T)_\sigma^+$ with $\bar{\mu}\precneqq \bar{\lambda}$, the Schubert cell $\Gr_\mathscr{G}^{\bar{\mu}}$ is contained in the singular locus of $\overline{\Gr}_\mathscr{G}^{\bar{\lambda}} $, except when  $\bar{\mu} \prec \bar{\lambda} $ is a cover relation and the simple short root $\gamma_\ell$  appears in $\bar{\lambda}-\bar{\mu}$.  We will prove this fact by several steps. Let $c_\ell$ be the coefficient of $\gamma_\ell$ in $\bar{\lambda}-\bar{\mu}$.  

Step 1.   Observe that using Lemma \ref{lem_iota_coroot} and by the same proof of Theorem \ref{thm_locus}, when the coefficient $c_\ell$  is even, we have $\dim D(1,\bar{\lambda})_{-\iota(\bar{\mu} )}\geq  2$.  Thus, $e_{\bar{\mu}}$ is singular in $\overline{\Gr}_\mathscr{G}^{\bar{\lambda}} $.

Step 2.   Assume that the coefficient $c_\ell>1$ and $c_\ell$ is odd.  There exists sequence of dominant elements in $X_*(T)^+_\sigma$,
\be 
\label{cover_rel}
 \bar{\mu}=  \bar{\lambda}_k  \prec  \bar{\lambda}_{k-1} \prec \cdots  \prec  \bar{\lambda}_1 \prec  \bar{\lambda}_{0} = \bar{\lambda} , \ee
such that each $\prec$ is a cover relation.  Then, by a theorem of Stembridge \cite[Theorem 2.8]{St},  for each $i$,   $  \bar{\lambda}_{i} - \bar{\lambda}_{i+1}  $ is a positive root of $H$, for any $0\leq  i \leq  k-1$, and the coefficient of $\gamma_\ell$ in each $ \bar{\lambda}_{i} - \bar{\lambda}_{i+1}$ is either 0 or 1.  Let $j$ be the least integer such that the coefficient of $\gamma_\ell$ in $\bar{\lambda}_{j-1}-\bar{\lambda}_j$ is 1. Such $j$  exists, since $c_\ell\not=1$.
  Then the coefficient of $\gamma_\ell$ in $\bar{\lambda}_j- \bar{\mu}$ is even.  By Step 1, we have  $\dim D(1,\bar{\lambda}_j )_{-\iota(\bar{\mu} )}\geq  2$. On the other hand, we have the inclusion $D(1,\bar{\lambda}_j )\subset D(1,\bar{\lambda})$. It follows that  $\dim D(1,\bar{\lambda})_{-\iota(\bar{\mu} )}\geq  2$. Hence, the variety $\overline{\Gr}^{\bar{\lambda}}_\mathscr{G}$ is singular at the point $e_{\bar{\mu}}$.

Step 3. We now assume that the coefficient $c_\ell=1$. By assumption, $\bar{\mu}\prec \bar{\lambda}$ is not a cover relation.   Then, in the sequence of cover relations in (\ref{cover_rel}),  either the coefficient of $\gamma_\ell$ in $  \bar{\lambda}_{k-1} - \bar{\lambda}_{k}  $ is 0, or the coefficient of $\gamma_\ell$ in   $  \bar{\lambda}_{0} - \bar{\lambda}_{1}  $ is 0.  
  If  the coefficient of $\gamma_\ell$ in $  \bar{\lambda}_{k-1} - \bar{\lambda}_{k}  $ is 0,  by Step 1 $\dim D(1,\bar{\lambda}_{k-1} )_{-\iota(\bar{\mu})}\geq 2$, implying that $\dim D(1,\bar{\lambda})_{-\iota(\bar{\mu})}\geq 2$. Hence $e_{\bar{\mu}}$ is singular in $\overline{\Gr}_\mathscr{G}^{\bar{\lambda}}$.
    If  the coefficient of $\gamma_\ell$ in $  \bar{\lambda}_{0} - \bar{\lambda}_{1}  $ is 0,  then by Step 1 again, $e_{\bar{\lambda}_1}$ is a singular point in $\overline{\Gr}_\mathscr{G}^{\bar{\lambda}}$.  Since the singular locus of  $\overline{\Gr}_\mathscr{G}^{\bar{\lambda}}$ is closed,  the point $e_{\bar{\mu}}$ is also singular in $\overline{\Gr}_\mathscr{G}^{\bar{\lambda}}$.

  We now explicitly describe the cover relation $\bar{\mu}\prec \bar{\lambda}$ such that $\gamma_\ell$ appears in $\bar{\lambda}-\bar{\mu}$.  Note that $X_*(T)_\sigma$ is a root lattice of $H\simeq {\rm SO}_{2n+1}$.  In fact, the lattice $X_*(T)_\sigma$ is spanned by $\varpi_1,\varpi_2, \cdots, \varpi_{\ell-1}, 2\varpi_\ell$.  Reading more carefully from \cite[Theorem 2.8]{St},  we can see that,  $\bar{\mu}\prec \bar{\lambda}$ is a cover relation and $\gamma_\ell$ appears in  $ \bar{\lambda}-\bar{\mu}$, if and only if one of the followings holds:
\begin{enumerate}
\item $ \bar{\lambda}-\bar{\mu}=\gamma_\ell$ and $\langle \bar{\mu},  \check{\gamma}_\ell \rangle\neq 0 $, where $\check{\gamma}_\ell$ is the coroot of $\gamma_\ell$.
\item  $ \bar{\lambda}-\bar{\mu}=\sum_{j=i}^\ell  \gamma_j$ and $\mu=\sum_{k=1}^{i-1 } a_k\varpi_k$, for some $1\leq i\leq \ell$.
\end{enumerate}

Let $G$ be the simply-connected simple group of type $A_{2\ell}$ with the standard automorphism $\sigma$ considered in this paper.  Let  $\alpha_1,\alpha_2,\cdots, \alpha_{2\ell}$ is a set of simple roots of $G$. Let $L$ be the Levi subgroup of $G$ generated by the simple roots \[ \alpha_{i}, \alpha_{i+1},\cdots, \alpha_{\ell}, \alpha_{\ell+1}, \cdots,   \alpha_{\ell+i} ,   \alpha_{\ell+i+1}.\]
  Let $M$ be the derived group $[L,L]$ of $L$. Then $M$ is simply-connected simple group of type $A_{ 2(\ell-i+1)}$ and $\sigma$ still acts on $M$ as a standard automorphism. Let $\mathscr{M}$ be the parahoric group scheme ${\rm Res}_{\mathcal{O}/\bar{\mathcal{O} }  } (M_\mathcal{O})^\sigma$, which is of type $A_{2(\ell-i+1)}^{(2)}$.  Let $T'=T\cap M$ be the maximal torus of $M$. We have the inclusion $X_*(T')\to X_*(T)$, and this induces an inclusion $X_*(T')_\sigma\to X_*(T)_\sigma$ and $X_*(T')^+_\sigma\to X_*(T)^+_\sigma$. We write $\bar{\lambda}=\sum_{k=1}^\ell  b_k \varpi_k$ and $\bar{\mu}=\sum_{k=1}^\ell  c_k \varpi_k$  with $b_k\geq 0, c_k\geq 0$ for any $k=1,\cdots, \ell$.  Set
\[ \bar{\lambda}' =\sum_{k=i}^{\ell} b_k \varpi_k, \quad  \bar{\mu}' =\sum_{k=i}^{\ell} c_k \varpi_k  ,\]
Then, $\bar{\lambda}', \bar{\mu}' \in X_*(T')^+_\sigma$. Moreover,  $\bar{\lambda}'- \bar{\mu}' =\bar{\lambda}-\bar{\mu}$.
 We have the following twisted analogue of Levi lemma,
\begin{equation} 
\label{lev_red_lem_MOV}   L^{--}\mathscr{G}\cdot e_{\bar{\mu}} \cap  \overline{\Gr}^{\bar{\lambda}}_{\mathscr{G}}  \simeq  L^{--}\mathscr{M}\cdot e_{\bar{\mu}' } \cap  \overline{\Gr}^{\bar{\lambda}' }_{\mathscr{M}}   , \end{equation}
where $ L^{--}\mathscr{G}$ and $L^{--}\mathscr{M}$ are defined as in (\ref{thm_4.2_eq}). It can be proved by exactly the same argument as in \cite[Corollary 3.4]{MOV}, relying on the following main ingredients:
\begin{enumerate}
\item $L^{--}\mathscr{G}\cdot e_{\bar{0}}$ (resp. $L^{--}\mathscr{M}\cdot e_{\bar{0}}$) is open in $\Gr_{\mathscr{G}}$ (resp. $\Gr_{\mathscr{M}}$ ), cf.\,(\ref{open_em_eq}).
\item Transversal slice lemma for $\Gr_\mathscr{G}$ and $\Gr_{\mathscr{M}}$ (cf.\,\cite[Lemma 2.5]{MOV}, \cite[Section 1.4]{KLu}), i.e. $ L^{--}\mathscr{G}\cdot e_{\bar{\mu}}$ (resp. $ L^{--}\mathscr{M}\cdot e_{\bar{\mu}' }$) is a transverse slice to $\overline{\Gr}^{\bar{\lambda}}_{\mathscr{G}}$ (resp. $ \overline{\Gr}^{\bar{\lambda}' }_{\mathscr{M}}$). 
\item $L^{--}\mathscr{G}\cdot e_{\bar{\mu}} \cap  \overline{\Gr}^{\bar{\lambda}}_{\mathscr{G}}$ and   $L^{--}\mathscr{M}\cdot e_{\bar{\mu}' } \cap  \overline{\Gr}^{\bar{\lambda}' }_{\mathscr{M}}  $ are reduced, irreducible and normal, cf.\,\cite[Lemma 2.6]{MOV}. Here, the main point is that $ \overline{\Gr}^{\bar{\lambda}}_{\mathscr{G}}$ and $ \overline{\Gr}^{\bar{\lambda}' }_{\mathscr{M}}$ are normal, cf.\,\cite[Theorem 0.3]{PR}.
\end{enumerate}
From the isomorphism (\ref{lev_red_lem_MOV}) and the transversal slice lemma, we have the following Levi reduction:   $e_{\bar{\mu}' }$ is singular in $\overline{\Gr}_{\mathscr{M}}^{\bar{\lambda}' }$ if and only if  $e_{\bar{\mu}}$ is singular in $\overline{\Gr}^{\bar{\lambda}}_{\mathscr{G}} $. 

Case (1): $ \bar{\lambda}-\bar{\mu}=\gamma_\ell$ and $\langle \bar{\mu},  \check{\gamma}_\ell \rangle\neq 0 $. In this case, we are reduced to $A_2^{(2)}$.  It is known from  \cite[Proposition 7.1]{HR}  that,  $e_{\bar{\mu}' }$ is singular in $\overline{\Gr}_{\mathscr{M}}^{\bar{\lambda}' }$, as $\bar{\lambda}'$ is not quasi-minuscule.  

Case (2):  $ \bar{\lambda}-\bar{\mu}=\sum_{j=i}^\ell  \gamma_j$ and $\mu=\sum_{k=1}^{i-1 } a_k\varpi_k$ with all $a_k>0$, for some $1\leq i\leq \ell$. In this case, $\bar{\mu}'=0$ and $\bar{\lambda}'=\sum_{j=i}^\ell  \gamma_j$. Thus, we are reduced to consider the singularity of quasi-minuscule affine Schubert variety $\overline{\Gr}_{\mathscr{M}}^{\bar{\lambda}' }$. It was observed by Richarz (using \cite[Prop 4.16]{Arz}) that the variety $\overline{\Gr}_{\mathscr{M}}^{\bar{\lambda}' }$ is smooth. We give a  different argument here. We consider the parahoric group scheme $\mathscr{G}={\rm Res}_{\mathcal{O}/\bar{\mathcal{O} }  } (G_\mathcal{O})^\sigma$, and let   $\overline{\Gr}_{\mathscr{G}}^ {\bar{\lambda}}$     be the quasi-minuscule Schubert variety. Let $\fg_{\mathrm{i}}$ be the eigenspace of $\sigma$ on the Lie algebra $\fg$ of eigenvalue $\mathrm{i}=\sqrt{-1}$. 
The vector space $\fg_{\mathrm{i}}$ consists of two $G^\sigma$-orbits, as $\fg_{\mathrm{i}}$ is actually the standard representation of $G^\sigma={\rm Sp}_{2\ell}$ which is of dimension $2\ell$. Thus, any element in $\fg_{\mathrm{i}}$ is nilpotent.   Then we consider a $G^\sigma$-equivariant embedding $\mathfrak{g}_{\mathrm{i}}\to  {\Gr}_{\mathscr{G}}$ given by $x\mapsto {\rm exp}(( {\rm ad } x)t^{-1}) \cdot e_{0}\in {\Gr}_{\mathscr{G}}$, where $e_0$ is the base point in $\Gr_{\mathscr{G}}$, and we regard ${\rm ad} x$ as an element in $G=G_{ad}$.  Since $e_\ell+ e_{\ell+1}\in \fg_\mathrm{i}$, and $\{ e_\ell+ e_{\ell+1}, f_\ell+f_{\ell+1}, h_{\ell}+h_{\ell+1} \}$ form a $sl_2$-triple, 
one may check easily that $\fg_\mathrm{i}$ is mapped into $\overline{\Gr}_{\mathscr{G}}^{\gamma_\ell}=\overline{\Gr}_{\mathscr{G}}^{\bar{\lambda}}$, in particular $0\mapsto e_0$. By comparing their dimensions and $G^\sigma$-equivariance, one may see this is an open embedding. Thus, $\overline{\Gr}_{\mathscr{G}}^{\bar{\lambda}}$ is smooth. 


\end{proof}
\begin{remark}
\label{absolutely_special_rmk}
When $\mathscr{G}$ is absolutely special of type $A_{2\ell}^{(2)}$, the smooth locus of twisted affine Schubert variety  $\overline{\Gr}_{\scrG}^{\bar{\lambda}} $  is the big cell. This was proved by Richarz in \cite{Ri2}. The idea is to use Levi reduction lemma of Malkin-Ostrik-Vybornov and Stembridge's combinatorial result \cite[Theorem 2.8]{St} to reduce to split rank one cases, in particular the case $A_{2}^{(2)}$ (a proof of this case also appears in \cite[Prop.7.1]{HR}), and the quasi-minuscule Schubert variety (a strong result of this case was proved by Zhu \cite{Zh2} that this variety is not Gorenstein). For the remaining cases, one can use non-triviality of Kazhdan-Lusztig polynomials, cf.\,\cite[Prop.6.4.3]{MOV}. 
\end{remark}
\begin{remark}
\label{Rmk_locus}
 One can define the affine Grassmannian $\Gr_{\scrG}$ and twisted affine Schubert varieties $\overline{\Gr}_{\scrG}^{\bar{\lambda}} $ of  the  special parahoric group scheme $\mathscr{G}$  with the base field $\mathrm{k}$ of characteristic $p$.  In \cite[Section 6]{HR}, when $p\not= r$, Haines and Richarz reduced the question of the smooth locus of the $\overline{\Gr}_{\scrG}^{\bar{\lambda}} $ over characteristic p to characteristic zero case.   In fact, by the work of Louren\c{c}o \cite{Lo}, one may construct a global twisted affine Schubert variety over $\mathbb{Z}$ so that the base change to the field $\mathrm{k}$ of characteristic $p$ (including $p=r$) is the given twisted affine Schubert variety defined over $\mathrm{k}$. 
\end{remark}

\section{Duality theorem for $E_6$}
\label{duality_E_6_sect}

In  \cite[Proposition 2.1.2]{Zh1}, Zhu showed that the duality theorem (cf.\,Theorem \ref{thm_zhu}) holds for any dominant coweight if the theorem holds for all fundamental coweights. For type $E_6$, Zhu was able to prove that the theorem holds for the fundamental coweights $\check{\omega}_1, \check{\omega}_2, \check{\omega}_3, \check{\omega}_5, \check{\omega}_6$ (Bourbaki labelling). However, the most difficult case $\check{\omega}_4$ remained open. In this section, we will prove that the theorem holds for $\check{\omega}_4$. Thereby, we complete the duality theorem for $E_6$ in general.

\subsection{Some reductions}
Let $G$ be a simply-laced simple group of adjoint type.  Let $T$ be a maximal torus in $G$.  Let $\mathrm{L}$ be the level one line bundle on $\Gr_G$. 
 For any $\lambda\in X_*(T)$, a general question is if the following restriction map is an isomorphism:
\begin{equation}
\label{duality_eq}
 H^0(\overline{\Gr}_{G}^{\lambda}, \mathrm{L}) \to  H^0((\overline{\Gr}_G^{\lambda} )^T, \mathrm{L}) . \end{equation}
 Zhu proved that this map is always surjective (cf.\,\cite[Prop.2.1.1]{Zh1}), and he also proved that the map is an isomorphism for type $A, D$ and many cases of $E$. 

\begin{theorem}
\label{thm_E_6}
The map (\ref{duality_eq}) is an isomorphism for any dominant coweight $\lambda$ when $G$ is of type $E_6$. 
\end{theorem}
Since Zhu has proved this theorem for $\check{\omega}_1, \check{\omega}_2, \check{\omega}_3, \check{\omega}_5, \check{\omega}_6$, by \cite[Prop.2.1.3]{Zh1} we will only need to prove that the theorem holds for $\check{\omega}_4$.

For convenience, we assume that $\lambda$ is dominant and $\lambda$ is in the coroot lattice, and we set 
\[  D(1,\lambda):=  H^0(\overline{\Gr}_{G}^{\lambda}, \mathrm{L} ) ^\vee, \quad   D^T(1,\lambda):= H^0((\overline{\Gr}_G^{\lambda} )^T, \mathrm{L}) ^\vee. \]
 Then, we can identify $D^T(1,\lambda)$ as a subspace of the affine Demazure module $D(1,\lambda)$.  
 
 Let $\tilde{L}(\fg)= \fg((t))\oplus \mathbb{C} K\oplus \mathbb{C} d$ be the affine Kac-Moody algebra associated to $\fg$ with center $K$ and scaling element $d$. Let $\mathscr{H}(\Lambda_0)$ denote the integrable highest weight representation of $\tilde{L}(\fg)$ of highest weight $\Lambda_0$. Let $v_0$ be the highest weight vector of  $\mathscr{H}(\Lambda_0)$.  For any $w\in W$, set
 \[v_{w(\lambda)}:= t^{w(\lambda)}\cdot v_0 . \]
  Then $v_{w(\lambda)}$ is an extremal vector in    $\mathscr{H}(\Lambda_0)$, and $\fh$-weight of $v_{w(\lambda)}$ is $-\iota(w(\lambda))$, where the map $\iota: X_*(T)\to \fh^\vee$ is induced by the normalized Killing form, cf.(\ref{iota_map}). By the theory of affine Demazure module (cf.\,\cite[Theorem 8.2.2 (a)]{Ku}), we have  
  \[  D(1,\lambda) = U(\fg[t] )\cdot v_{w(\lambda)} , \quad \text{ for any }w\in W. \]

 Given a Levi subgroup $L$ of $G$ and let $M$ be the derived group $[L,L]$. Let $\mathfrak{m}$ denote the Lie algebra of $M$, and denote the current algebra of $\mathfrak{m}$ by $\mathfrak{m}[t]$. By \cite[Corollary 1.3.8, Lemma 2.2.6]{Zh1}, we have the following Levi reduction lemma.
\begin{lemma}
\label{Levi_red_lem}
If the map (\ref{duality_eq}) is an isomorphism for $M$, then $U(\mathfrak{m}[t] )\cdot v_{w(\lambda)}$ is contained in $D^T(1,\lambda)$, for any element $w$ in the Weyl group $W$ of $G$. 
\end{lemma}

Let $\mu$ be a dominant coweight of $G$ such that $\mu \prec  \lambda$.  The following restriction map is surjective:
\[  H^0(\overline{\Gr}_{G}^{\lambda}, \mathrm{L}  ) \to H^0(\overline{\Gr}_{G}^{\mu}, \mathrm{L} )   .\]
Hence, it induces an inclusion $D(1,\mu)\subset D(1,\lambda)$.  The following lemma is easy. 
\begin{lemma}
\label{dom_wt_red_lem}
If the map (\ref{duality_eq}) is an isomorphism for $\mu$, then $D(1,\mu)$ is contained in $D^T(1, \lambda)$. 
\end{lemma} 

Let $N_G(T)$ denote the normalizer group of $T$ in $G$.  Then $N_G(T)$ acts on $(\overline{\Gr}_G^{\lambda} )^T$ and thus on the vector space $D^T(1,\lambda)$.  
Note that 
\[  (\overline{\Gr}_G^{\lambda} )^T=(\Gr_G)^T\cap \overline{\Gr}_G^{\lambda} \simeq \Gr_T \cap  \overline{\Gr}_G^{\lambda} , \] 
where the second isomorphism follows from Theorem \ref{thm_fixed_point1}.  Hence, the Lie algebra $\fh[t]$ acts on $D^T(1,\lambda)$.  \vspace{0.5cm}


\begin{notation}
Let $V(\eta, i)$ denote the irreducible representation of $\fg$ of highest weight $\eta$ and of degree $i$ with respect to the action of $d$, that appears in the affine Demazure module $D(1,\lambda)$. For any $\fh$-weight $\nu$, we write $V(\eta, i)_{\nu}$ for the $\nu$-weight space of this representation. 
\end{notation}
We now consider the case when $G$ is of type $E_6$ and $\lambda=\check{\omega}_4$.  
All dominant coweights dominated by $\check{\omega}_4$ are described as follows,
\begin{equation}
\label{order_omega_4}
 0\prec \check{\omega}_2 \prec \check{\omega}_1+\check{\omega}_6\prec  \check{\omega}_4.     \end{equation}
 
 For convenience, we set 
\[ v_{\omega_4}:= v_{w_0(\check{\omega}_4 )}.\]
Then $v_{\omega_4}$ is an extremal weight vector in $\mathscr{H}(\Lambda_0)$ whose $\fh$-weight is $\omega_4$, since $-w_0(\omega_4)=\omega_4$  and $\iota(\check{\omega}_4 )=\omega_4$. The Demazure module $D(1, \check{\omega}_4)$ contains $V(\omega_4, -3)$, and $v_{\omega_4}$ is the highest weight vector of $V(\omega_4, -3)$.  By \cite[Section 3]{Kl}, we have the following decomposition 
\begin{equation}
D(1, \check{\omega}_4)= V(0,0) \oplus V(\omega_2,-1) \oplus V(\omega_1+\omega_6,-2) \oplus V(\omega_2,-2) \oplus V(\omega_4,-3).
\end{equation}

Since Zhu has proved that the map (\ref{duality_eq}) is an isomorphism for $\check{\omega}_1$ and $\check{\omega}_6$, hence also for $\check{\omega}_1+\check{\omega}_6$ (cf.\,\cite[Prop.2.1.3]{Zh1}), by Lemma \ref{dom_wt_red_lem} we have $D(1, \check{\omega}_1+\check{\omega}_6)\subset D^T(1,\check{\omega}_4)$.   Moreover, it is easy to see that
\[ D(1, \check{\omega}_1+\check{\omega}_6) =  V(0,0) \oplus V(\omega_2,-1) \oplus V(\omega_1+\omega_6,-2) .  \]
It follows that, $V(0,0) \oplus V(\omega_2,-1) \oplus V(\omega_1+\omega_6,-2) $ is contained in $D^T(1, \check{\omega}_4)$.  Thus, it suffices to show that $V(\omega_2,-2)$ and   $V(\omega_4,-3) $ are also contained in $D^T(1, \check{\omega}_4)$.  Since $D(1, \check{\omega}_4 )$ is $N_G(T)$-stable it can be further reduced to show that for any dominant weight $\nu$ of $\fg$,  the weight space $V(\omega_2,-2)_\nu$ and $V(\omega_4,-3)_\nu$ are contained in $D^T(1,\check{\omega}_4 )$.  In the remaining part of this section, we will analyze case by case and show that it is indeed true.

\subsection{ The representation $V(\omega_4, -3)$}

The dominant character of $V(\omega_4)$ is $e^{\omega_4}+4e^{\omega_1+\omega_6}+15e^{\omega_2}+45e^{0}$. By Lemma \ref{Levi_red_lem}, we have $v_{w(\check{\omega}_4 )}\in D^T(1, \check{\omega}_4 )$ for any $w\in W$. We conclude that $V(\omega_4,-3)_{\omega_4}\subset D^T(1,\check{\omega}_4 )$.




\subsubsection{ The weight space $V(\omega_4,-3)_{\omega_1+\omega_6}$}

In terms of simple roots, we have $\omega_4= 2 \alpha_1 +3\alpha_2+4\alpha_3+6\alpha_4+4\alpha_5+2\alpha_2$ and $\omega_1+\omega_6= 2 \alpha_1+2 \alpha_2 +3\alpha_3+4\alpha_4+3\alpha_5+2\alpha_6$. Thus the difference $\omega_4-(\omega_1+\omega_6)=\alpha_2+\alpha_3+2\alpha_4+\alpha_5$; in other words, this difference is supported on the Levi of type $D_4$ with simple roots $\alpha_2, \alpha_3, \alpha_4, \alpha_5$. By applying Chevalley generators $f_2, f_3,f_4,f_5$ on the highest weight vector $v_{\omega_4}$, we can get a spanning set of the weight space $V(\omega_4, -3)_{ \omega_1+\omega_6 }$. By Lemma \ref{Levi_red_lem}, we have 
\[V(\omega_4, -3)_{ \omega_1+\omega_6 }\subset  D^T(1, \check{\omega}_4 ). \]

\subsubsection{The weight space $V(\omega_4,-3)_{\omega_2}$} 
This case requires some brute force. We have the following difference: $\beta=\omega_4-\omega_2=\alpha_1+\alpha_2+2\alpha_3+3\alpha_4+2\alpha_5+\alpha_6$, whose height is 10.

We actually consider all expressions of the form $f_{{i_1}} \dots f_{{i_{10}}} v_{\omega_4}$ such that this vector is of weight $\omega_2$; this provides a spanning set of vectors in $V(\omega_4,-3)_{\omega_2}$ (with many relations!).

\begin{definition}
	We say a nonzero vector of the form $f_{{i_1}} f_{i_2}\dots f_{{i_{10}}} v_{\omega_4}$ is Levi-extremal, if there exists $ 1\leq k\leq 10$, such that
	\begin{enumerate}
\item
$f_{i_1},f_{i_2},\cdots, f_{i_{k}}$ are contained in a proper Levi subalgebra of $\fg$;
\item  $f_{i_{k+1}}\cdots f_{i_{10}} v_{\omega_4}$ is an extremal vector in $V(\omega_4,-3)$. 
\end{enumerate}
\end{definition}
Observe that any proper Levi subalgebra of $E_6$ is either of type $A$ (or their product),  or of type $D$, and the restriction map (\ref{duality_eq}) for these types is always an isomorphism, which is due to Zhu. Then, to show that $V(\omega_4,-3)_{\omega_2}\subset D^T(1,\check{\omega}_4 )$, by Lemma \ref{Levi_red_lem} it suffices to prove the following proposition.
\begin{proposition}
\label{prop_levi_extemal}
Any nonzero vector $f_{{i_1}} f_{i_2}\dots f_{{i_{10}}} v_{\omega_4}$ of weight $\omega_2$ is Levi-extremal. 
\end{proposition}




\begin{proof}

Before we prove this proposition,  we first describe a poset of weights $\mu$ with partial order $<$,  such that $\mu< \omega_4$,  and $\mu-\omega_2$ is a sum of positive roots. For any two weights $\mu,\mu'$ in the poset, we write $\mu\xrightarrow{s_i}\mu'$ if $\mu'=s_i(\mu)$ and $\langle \mu, \check{\alpha}_i  \rangle\geq 1$. The partial order $<$ of this poset is generated by these simple relations. The weight $\mu$ will be labelled by $\textcolor{red}{*}$ once the support of $\mu-\omega_2$ (as a linear combination of simple roots) is contained in a proper sub-diagram of the Dynkin diagram of  $E_6$. We won't describe those weights below the $\textcolor{red}{*}$-labelled weights.  In the following two figures, we describe this poset by representing weights respectively in  terms of the coordinates with respect to fundamental weights and simple roots. We have the following rules.
\begin{enumerate}
\item In the first figure, if $\mu\xrightarrow{s_i}\mu'$, then the number at vertex $i$ decreases by 2, and the adjacent vertices increase by 1;
\item In the second figure, if $\mu\xrightarrow{s_i}\mu'$, then the number at vertex $i$ decrease by 1, and no changes elsewhere. 
\end{enumerate}

\begin{figure}[h!]
	\centering
	\begin{tikzcd}[scale cd=0.75]
		& & \begin{smallmatrix} 0 & 0 & 1 & 0 & 0\\ & & 0 & & \\ \end{smallmatrix} \arrow[d,"s_4"] &  & \\ 
		& & \begin{smallmatrix} 0 & 1 & -1 & 1 & 0 \\ & & 1 & & \\ \end{smallmatrix} \arrow[dl, "s_3"] \arrow[d,"s_2"] \arrow[dr, "s_5"] & & \\
		& \begin{smallmatrix} 1 & -1 & 0 & 1 & 0 \\  & & 1 & & \\ \end{smallmatrix} \arrow[dl, "s_1"] \arrow[d, "s_2"]\arrow[dr,"s_5"] & \begin{smallmatrix} 0 & 1 & 0 & 1 & 0 \\ & & -1 &\textcolor{red}{*} &  \end{smallmatrix}  & \begin{smallmatrix} 0 & 1 & 0 & -1 & 1 \\ & & 1 & & \\ \end{smallmatrix} \arrow[dl, "s_3"] \arrow[d,"s_2"] \arrow[dr,"s_6"] & \\		
		\begin{smallmatrix} -1 & 0 & 0 & 1 & 0 \\ & & 1 &\textcolor{red}{*} &  \end{smallmatrix} & \begin{smallmatrix} 1 & -1 & 1 & 1 & 0 \\ & & -1 &\textcolor{red}{*} &  \end{smallmatrix} & \begin{smallmatrix} 1 & -1 & 1 & -1 & 1 \\ & & 1 & & \\ \end{smallmatrix}   \arrow[d, "s_4"] \arrow[drr, "s_6"] \arrow[dl, "s_2"]
		\arrow[dll, "s_1"]
		& \begin{smallmatrix} 0 & 1 & 1 & -1 & 1 \\ & & -1 & \textcolor{red}{*} & \\ \end{smallmatrix}
		& \begin{smallmatrix} 0 & 1 & 0 & 0 & -1 \\ & & 1 & \textcolor{red}{*} & \end{smallmatrix} \\
		\begin{smallmatrix} -1 & 0 & 1 & -1 & 1 \\ & & 1 & \textcolor{red}{*} & \\ \end{smallmatrix}  & \begin{smallmatrix} 1 & -1 & 2 & -1 & 1 \\ & & -1 &\textcolor{red}{*} & \end{smallmatrix} & \begin{smallmatrix} 1 & 0 & -1 & 0 & 1 \\ & & 2 & & \end{smallmatrix} \arrow[dl, "s_1"] \arrow[dr, "s_6"] & & \begin{smallmatrix} 1 & -1 & 1 & 0 & -1 \\ & & 1 & \textcolor{red}{*} &  \end{smallmatrix} \\
		& \begin{smallmatrix} -1 & 1 & -1 & 0 & 1 \\ & & 2 &\textcolor{red}{*} &  \end{smallmatrix} & & \begin{smallmatrix} 1 & 0 & -1 & 1 & -1 \\ & & 2 & \textcolor{red}{\textcolor{red}{*}} &  \end{smallmatrix}  & 
	\end{tikzcd}
	\caption{  In this diagram, the weight $\mu$ is represented by the coordinates of $\mu$ with respect to fundamental weights. When 1 occurs at vertex $i$, it indicates that we can apply reflection $s_i$. }

	\begin{tikzcd}[scale cd=0.75]
		& & \begin{smallmatrix} 1 & 2 & 3 & 2 & 1\\ & & 1 & & \\ \end{smallmatrix} \arrow[d,"s_4"] &  & \\ 
		& & \begin{smallmatrix} 1 & 2 & 2 & 2 & 1 \\ & & 1 & & \\ \end{smallmatrix} \arrow[dl, "s_3"] \arrow[d,"s_2"] \arrow[dr, "s_5"] & & \\
		& \begin{smallmatrix} 1 & 1 & 2 & 2 & 1 \\  & & 1 & & \\ \end{smallmatrix} \arrow[dl, "s_1"] \arrow[d, "s_2"]\arrow[dr,"s_5"] & \begin{smallmatrix} 1 & 2 & 2 & 2 & 1 \\ & & 0 &\textcolor{red}{*} &  \end{smallmatrix}  & \begin{smallmatrix} 1 & 2 & 2 & 1 & 1 \\ & & 1 & & \\ \end{smallmatrix} \arrow[dl, "s_3"] \arrow[d,"s_2"] \arrow[dr,"s_6"] & \\		
		\begin{smallmatrix} 0 & 1 & 2 & 2 & 1\\ & & 1 &\textcolor{red}{*} &  \end{smallmatrix} & \begin{smallmatrix} 1 & 1 & 2 & 2 & 1 \\ & & 0 &\textcolor{red}{*} &  \end{smallmatrix} & \begin{smallmatrix} 1 & 1 & 2 & 1 & 1 \\ & & 1 & & \\ \end{smallmatrix}   \arrow[d, "s_4"] \arrow[drr, "s_6"] \arrow[dl, "s_2"]
		\arrow[dll, "s_1"]
		& \begin{smallmatrix} 1 & 2& 2 & 1 & 1 \\ & & 0 & \textcolor{red}{*} & \\ \end{smallmatrix}
		& \begin{smallmatrix} 1 & 2 & 2 & 1 & 0 \\ & & 1 & \textcolor{red}{*} & \end{smallmatrix} \\
		\begin{smallmatrix} 0 & 1 & 2 & 1 & 1 \\ & & 1 & \textcolor{red}{*} & \\ \end{smallmatrix}  & \begin{smallmatrix} 1 & 1 & 2 & 1 & 1 \\ & & 0 &\textcolor{red}{*} & \end{smallmatrix} & \begin{smallmatrix} 1 & 1 & 1 & 1 & 1 \\ & & 1 & & \end{smallmatrix} \arrow[dl, "s_1"] \arrow[dr, "s_6"] & & \begin{smallmatrix} 1 & 1 & 2 & 1 & 0 \\ & & 1 & \textcolor{red}{*} &  \end{smallmatrix} \\
		& \begin{smallmatrix} 0 & 1 & 1 & 1 & 1 \\ & & 1 &\textcolor{red}{*} &  \end{smallmatrix} & & \begin{smallmatrix} 1 & 1 & 1 & 1 & 0 \\ & & 1 & \textcolor{red}{*} &  \end{smallmatrix}  & 
	\end{tikzcd}
	\caption{This is the same diagram as in Figure 1. The difference is that, the weight $\mu$ is represented by the coordinates of $\mu-\omega_2$  with respect to simple roots.  This diagram tells when the support $\mu-\omega_2$ is contained in a proper subdiagram.}
\end{figure}

We first show that any nonzero vector $f_{{i_7}}f_{{i_8}}f_{{i_9}}f_{{i_{10}}}v_{\omega_4}$ of weight $\mu$ is an extremal weight vector. This can be easily checked from the first figure, since no integer $\geq 2$ appears as a coefficient of $\omega_i$ until the 4th step at least. As a result, $f_{{i_7}}f_{{i_8}}f_{{i_9}}f_{{i_{10}}}v_{\omega_4}=s_{{i_7}}s_{{i_8}}s_{{i_9}}s_{{i_{10}}}v_{\omega_4}$ for any nonzero vector $f_{{i_7}}f_{{i_8}}f_{{i_9}}f_{{i_{10}}}v_{\omega_4}$.
 Now it is clear from the second figure that,  any nonzero vector $f_{i_1}\cdots  f_{{i_7}}f_{{i_8}}f_{{i_9}}f_{{i_{10}}}v_{\omega_4}$ is Levi-extremal, if 
\[ f_{i_6} f_{{i_7}}f_{{i_8}}f_{{i_9}}f_{i_{10}} \neq  f_4f_3f_5f_4=f_4f_5f_3f_4.\]

Thus, the ``worst possible" case, from the perspective of producing Levi-extremal vectors, has the first four lowering operators as follows: $f_{4}f_{5}f_{3}f_{4}  v_{\omega_4}=f_4f_5f_3f_4 v_{\omega_4}$. All other nontrivial applications of 4 lowering operators will result in a Levi-extremal vector or the $0$ vector. We further observe that both $f_{1}f_4 f_{5}f_{3}f_{4}v_{\omega_4}$ and $f_{6}f_{4}f_{5}f_{3}f_{4}v_{\omega_4}$ result in Levi-extremal vectors.


The only remaining vector to consider is $f_{2}f_{4}f_{5}f_{3}f_{4}v_{\omega_4}$. Note that this is the first case where $s_{2}s_{4}s_{5}s_{3}s_{4}v_{\omega_4} \neq f_{1}f_{4}f_{5}f_{3}f_{4}v_{\omega_4}$, since $\langle \check{\alpha}_2, \omega_4-2\alpha_4-\alpha_3-\alpha_5 \rangle =2$.   Thus every element $f_{{i_1}} \dots f_{{i_{10}}} v_{\omega_4}$ is Levi-extremal,  except those of which the first five lowering operators are precisely $f_{2}f_{4}f_{5}f_{3}f_{4} $ (up to the order of $f_3$ and $f_5$); what remains are the lowering operators $f_1,f_3,f_4,f_5,f_6$. By the same logic, the next lowering operator must be $f_4$, since any other lowering operator would commute with $f_2$, returning us to the Levi-extremal vectors situation. Thus we are left with is considering the element $f_{4}f_{2}f_{4}f_{5}f_{3}f_{4}v_{\omega_4}$.

Note that $f_{\alpha_2+\alpha_4}:=f_{\alpha_4}f_{\alpha_2}-f_{\alpha_2}f_{\alpha_4}$ is a root vectorof root $\alpha_2+\alpha_4$.  Since $f_4f_{4}f_{5}f_{3}f_{4}v_{\omega_4}=0$, we have
\[f_{4}f_{2}f_{4}f_{5}f_{3}f_{4}v_{\omega_4}=f_{\alpha_2+\alpha_4}f_{4}f_{5}f_{3}f_{4}v_{\omega_4}  .\]
This is the extremal vector $s_{\alpha_2+\alpha_4}s_{4}s_{5}s_{3}s_{4}v_{\omega_4}  $, since the weight of the extremal weight vector $f_{\alpha_4}f_{\alpha_5}f_{\alpha_3}f_{\alpha_4}v_{\omega_4}$ is $\omega_1+2\omega_2-\omega_4+\omega_6$, and the pairing $\langle  \omega_1+2 \omega_2-\omega_4+\omega_6,\check{\alpha}_2+\check{\alpha}_4 \rangle =1$. Our remaining lowering operators are $f_1,f_3,f_5,f_6$. They are contained in a proper Levi subalgebra.   Thus, any nonzero vector of the form
\[ f_{i_1}f_{i_2}f_{i_3}f_{i_4}  f_{4}f_{2}f_{4}f_{5}f_{3}f_{4}v_{\omega_4}  \]
is always Levi-extremal. This concludes the proof.

\end{proof}


\begin{remark}
The construction of the above figures comes from the ``Numbers game", due to Proctor (unpublished) and explored in Mozes \cite{Mo} and Proctor \cite{Pro}. When the representation is minuscule, the action of the simple reflections is described by the algorithm above; adding one to adjacent nodes and subtracting two from the given node (precisely the action of subtracting a simple root in the simply-laced types). While the representation $V(\omega_4)$ is not minuscule or quasi-minuscule, these techniques still proved useful in this case and could be useful for the study of the remaining fundamental representations of $E_7$ and $E_8$, where the restriction isomorphism is not yet known.
\end{remark}

\subsubsection{The weight space $V(\omega_4,-3)_0$} 
\label{sect_5.2.3}

Let $\theta$ be the highest root of $\fg$, and let $\beta= \omega_4-\omega_2=\alpha_1+\alpha_2+2\alpha_3+3\alpha_4+2\alpha_5+ \alpha_6$. $\beta$ is also a positive root of $\fg$.   We consider the element $f_\theta  f_{\beta} v_{\omega_4}$. 
The following proposition is verified by Travis Scrimshaw using \textsc{SageMath}. See Appendix \ref{appendix_sect}. 
\begin{proposition}
\label{Travis_prop}
The $W$-span of $f_\theta  f_{\beta} v_{\omega_4}$ is the weight zero space  $V(\omega_4,-3)_0$. 
\end{proposition}

One can check that $\alpha_1, \alpha_3,\alpha_4,\alpha_5, \theta, -\beta$ form a system of simple positive roots of $E_6$, and $\theta, -\beta$ form a subsystem of type $A_2$.  Then by Lemma \ref{Levi_red_lem}, $f_{\theta}f_\beta  v_{\omega_4} \in D^T(1,\check{\omega}_4)$. By $W$-invariance on $D^T(1,\check{\omega}_4)$ and the above proposition, we can conclude that 
\begin{equation}
\label{0_weight_incl}
V(\omega_4,-3)_0\subset  D^T(1,\check{\omega}_4)  .\end{equation}

 \subsection{The representation $V(\omega_2, -2)$}

First we consider the ``0-string" of the full basic representation $\mathscr{H}(\Lambda_0)$; this is the direct sum $\oplus_{n\geq 0}\mathscr{H}(\Lambda_0)_{-n\delta}$. The Weyl group $W$ acts on each of these weight spaces, so $\mathscr{H}(\Lambda_0) _{-n\delta}$ is a direct sum of irreducible representations of the Weyl group $W$. We first describe  $\mathscr{H}(\Lambda_0)_{-n\delta}$ for $n=0,1,2,3$, as representations of $W$.  By \cite[Proposition 12.13]{Ka}, we have the following decompositions:
\[ \mathscr{H}(\Lambda_0)_{0}= \mathbb{C} v_0\simeq  \mathbb{C} , \]
\[  \mathscr{H}(\Lambda_0)_{-\delta}= \fh t^{-1}\cdot v_0 \simeq   \fh \]
\begin{equation}  
\label{dec_eq_1}
\mathscr{H}(\Lambda_0)_{-2\delta}= \fh t^{-1} \cdot \fh t^{-1}\cdot v_0\oplus  \fh t^{-2}\cdot v_0\simeq   S^2\fh \oplus \fh \end{equation}
\begin{equation} 
\label{dec_eq_2}
 \mathscr{H}(\Lambda_0)_{-3\delta}=\fh t^{-1} \cdot  \fh t^{-1} \cdot \fh t^{-1}\cdot v_0\oplus  \fh t^{-1} \cdot  \fh t^{-2}\cdot v_0  \oplus   \fh t^{-3}\cdot v_0  \simeq  S^3\fh  \oplus  T^2 \fh \oplus   \fh. \end{equation}

The weight space  $V(\omega_4,-3)_0$ is a subrepresentation of $\mathscr{H}(\Lambda_0)_{-3\delta}$ with respect to the action of $W$.   It is known that $V(\omega_4,-3)_0$ is a direct sum of two irreducible W-representations, one 15 dimensional and the other is 30 dimensional, cf.\,\cite[Table 5,p.24]{AH}.  We will denote these subrepresentations by $\Pi_{15}$ and $\Pi_{30}$ indexed by their dimensions.

\begin{lemma}
\label{irred_15_lem}
The subspace $\Pi_{15}$ is the exactly the span of vectors 
\[ (h[t^{-1}]h'[t^{-2}]- h'[t^{-1}]  h[t^{-2}] )\cdot v_0 , \quad \text{ for all } h, h'\in \fh. \] 
\end{lemma}

\begin{proof}
First of all, we observe that $\Pi_{30}$ is contained in $S^3\fh$; this follows from a dimension check since $\dim( \mathfrak{h})=6$ and $T^2 \mathfrak{h} \simeq S^2 \mathfrak{h} \oplus \wedge^2 \mathfrak{h}$, this decomposition is compatible with the $W$-module structure, and the dimensions of each summand are both less than 30.  Secondly, we will prove that $S^3\fh$ doesn't contain a 15-dimensional irreducible representation. 

 Suppose that it is not the case. From the character table of the Weyl group of $E_6$ (cf.\,\cite[p.415]{Car}), we know that the dimensions of irreducible representations of $W$ are 1,6,15,20,24, etc. Then $S^3\fh$ must decompose into a 15-dimensional module and then either a 6-dimensional irreducible and 5 1-dimensional irreducibles or a 15-dimensional and 11 1-dimensional irreducibles. 
 
 Both of these options are impossible for the following reasons. The only two one-dimensional representations of $W$ are the trivial representation and the sign representation. The trivial representation cannot appear in $S^3 \mathfrak{h}$,  since this would give a $W$-invariant degree 3 polynomial on $\fh\simeq \fh^\vee$. This is impossible, because the possible degrees of invariant polynomials are 2,5,6,8,9,12; this list is the set of exponents +1 which can be found in \cite[p.231]{Bo}.

The other option is that all of these 1-dimensional irreducibles are the sign representation. However, we have the following decomposition as representation of $\fg$:
\[ \mathscr{H}(\Lambda_0)_{-3} = V(\omega_4) \oplus k_1 V(\omega_1+\omega_6) \oplus k_2 V(\omega_2) \oplus k_3 V(0)\]
for certain multiplicities $k_1,k_2,k_3$, where $ \mathscr{H}(\Lambda_0)_{-3}$ denote the degree $-3$ part of $\mathscr{H}(\Lambda_0)$ withe respect to the action of $d$. From \cite[Table 5,p.24]{AH}, one can see that no sign representation appear in the weight zero space of these irreducible representations. Thus, this option is also impossible. 

Therefore, $\Pi_{15}$ must be contained in $T^2\fh$. Note that $T^2\fh=S^2\fh \oplus \wedge^2\fh$.  Moreover, $V(\omega_1+\omega_6)_0$ and $V(\omega_2)_0$ are contained in $\mathscr{H}(\Lambda)_{-2\delta}$.  From the decomposition (\ref{dec_eq_1}) and \cite[Table 5,p.24]{AH}, we know that $S^2\fh$ is decomposed as a direct sum of a 20-dimensional irreducible and a 1-dimensional trivial representation. Thus, $\Pi_{15}$ is exactly the subspace $\wedge^2\fh$. In other words,   $\Pi_{15}$
is exactly the span of all vectors $(h[t^{-1}] h' [t^{-2}]- h'[t^{-1}] h[t^{-2}] )\cdot v_0 $, $h, h'\in \fh$. 
\end{proof}

\subsubsection{The weight space $V(\omega_2, -2)_0 $}
We have  $V(\omega_2, -2)_0 =\fh t^{-2}\cdot v_0 $, which is an irreducible representation of dimension 6. 

We choose any two nonzero elements $h, h'$ in $\fh$ such that $(h|h)=1$ and $(h|h')=0$, where $(\cdot | \cdot )$ is the normalized Killing form on $\fg$. 
Then 
\begin{equation}
\label{5.3.1_eq1}
  h[t]   (h[t^{-1}]  h'[t^{-2}]- h'[t^{-1}] h[t^{-2}] )\cdot v_0    =  h'[t^{-2}] \cdot v_0\in V(\omega_2, -2)_0  .\end{equation}

This is a nonzero vector. By the inclusion (\ref{0_weight_incl}) in Section \ref{sect_5.2.3} and Lemma \ref{irred_15_lem},   $(h[t^{-1}]  h'[t^{-2}]- h'[t^{-1}]  h[t^{-2}] )\cdot v_0 \in D^T(1,  \check{\omega}_4)$.  Since $D^T(1,\check{\omega}_4)$ is stable under the action of $\fh[t]$,  by (\ref{5.3.1_eq1} ) we have $h't^{-2}\cdot v_0\in D^T(1, \check{\omega}_4)$. Since $V(\omega_2, -2)_0$ is an irreducible representation of $W$ and  $D^T(1, \check{\omega}_4)$ is $W$-invariant, we get
\[V(\omega_2, -2)_0\subset D^T(1,\check{\omega}_4). \]

\subsubsection{The weight space $V(\omega_2,-2)_{\omega_2}$}
 We choose $h_1=\check{\alpha}_1$ and $h_2=\check{\alpha}_2$ in $\fh$. Then $(h_1|h_2)=0$ and $(h_1|h_1)=2$. By Lemma \ref{irred_15_lem}, we may  consider the following element 
 \[( h_1[t^{-1}]  h_2[t^{-2}]-h_2[t^{-1}]  h_1[t^{-2}] )\cdot v_0\]
in  $V(\omega_4, -3)_0$. 
Set  
\[  u:= e_\theta( h_1[t^{-1}]  h_2[t^{-2}]-h_2[t^{-1}] h_1[t^{-2}] )\cdot v_0 .\]
 This is an element in $V(\omega_4, -3)_{\theta}$. Note that $\theta=\omega_2$.
One may compute easily and get 
\[ u= ( h_1[t^{-1}]  e_\theta [t^{-2}]- e_\theta[ t^{-1}]  h_1[t^{-2}] )\cdot v_0    . \]
Then we have the following,
\[h_1[t]\cdot u=  2 e_\theta [t^{-2}] \cdot v_0  .\]
Now, it is easy to see that $h_1[t]\cdot u$ is nonzero and is a highest weight vector of $\fg$ of weight $\omega_2$.  Thus, we have shown that 
\[  V(\omega_2,-2)_{\omega_2} \subset D^T(1, \check{\omega}_4) . \]

Thus we may conclude that $D^T(1,\check{\omega}_4)=D(1,\check{\omega}_4)$. This finishes the proof of Theorem \ref{thm_E_6}.






\appendix
 \section{Proof of Proposition \ref{Travis_prop}, by Travis Scrimshaw}
 \label{appendix_sect}

    We will prove Proposition \ref{Travis_prop}  by using \textsc{SageMath}, which asserts that the $W$-span of $f_\theta f_\beta v_{\omega_4}$  is the weight zero space $V(\omega_4)_0$ of the fundamental representation $V(\omega_4)$ of $E_6$. 

\subsection{Lie algebra representations and crystals}

We briefly review some basic material on finite dimensional simple Lie algebras and their finite dimensional highest weight representations.
For more information, we refer the reader to~\cite{FH91}.
Let $\field$ be an algebraically closed field of characteristic $0$.
In this appendix, we restrict to the case when $G$ is a simple Lie group, and we typically consider the case $\field = \CC$.

By looking at the tangent space of the identity, we have a finite dimensional simple Lie algebra $\g$ over $\field$ that is generated by $E_i, F_i, H_i$ for $i \in I$ with the relations
\[
\begin{array}{c@{\qquad}c@{\qquad}c}
[H_i, H_j] = 0,
& [E_i, F_j] = \delta_{ij} H_i,
\\ \mbox{} [H_i, E_j] = \inner{\alpha_i}{\coroot_j} E_j,
& [H_i, F_j] = -\inner{\alpha_i}{\coroot_j}F_j,
\\ \ad(E_i)^{-\inner{\alpha_i}{\coroot_j}+1} E_j = 0
& \ad(F_i)^{-\inner{\alpha_i}{\coroot_j}+1} F_j = 0 & (i \neq j),
\end{array}
\]
where $\ad(X) Y = [X, Y]$ is the adjoint operator.
Let $\h = \Span_{\field} \{h_i\}_{i \in I}$ denote the Cartan subalgebra corresponding to $T$.

A representation of a Lie algebra $V$ is a vector space over $\field$ such that $[X, Y] v = X(Yv) - Y(Xv)$ for all $v \in V$. For two $\g$-representations $V$ and $W$, their tensor product is naturally a $\g$-representation by
\[
X(v \otimes w) = Xv \otimes w + v \otimes Xw
\]
for all $X \in \g$ and $v \otimes w \in V \otimes W$.
We restrict to the category of finite dimensional highest weight representations, and we let $V(\lambda)$ denote the irreducible highest weight representation for the dominant integral weight $\lambda \in P^+$.
The Weyl group action on $V$ given by
\[
s_i = \exp(F_i) \exp(-E_i) \exp(F_i),
\]
%
For any nilpotent element $X\in \fg$ and $v \in V$, we can implement $\exp(X) v$ by finding $K = \min \{k \in \ZZ_{>0} \mid X^kv \neq 0\}$ and then computing
\[
\exp(X) v = \sum_{k=0}^K \frac{X^k v}{k!}.
\]

We give an explicit realization for a minuscule representation following the construction in \cite[Sec.~3.1.1]{OS19}, where we prove the analog of~\cite[Prop.~3.2, Prop.~3.3]{OS19} for $\g$-representations.
For the remainder of this section, let $r \in I$ be such that $\inner{\fw_r}{\coroot} \leq 1$ for all $\alpha \in \Phi^+$, which characterizes the minuscule representations $V(\fw_r)$.

A \defn{crystal} for $\g$ is a set $B$ with crystal operators $\te_i, \tf_i \colon B \to B \sqcup \{0\}$, for all $i \in I$, that satisfy certain properties and encode the action of the Chevalley generators $E_i$ and $F_i$ respectively.
Kashiwara showed~\cite{Kas90,Kas91} that all highest weight representations $V(\lambda)$ have corresponding crystals $B(\lambda)$.
We denote by $u_{\lambda}$ the unique highest weight element of $B(\lambda)$.
For more information on crystals, we refer the reader to~\cite{BS17}.

For a minuscule node $r$, let $J := I \setminus \{r\}$, and let $W_J := \langle s_i \mid i \in J \rangle$ denote the corresponding subgroup.
Denote by $W^J$ the set of minimal length coset representatives of $W / W_J$.
Define crystal operators $\te_i, \tf_i \colon W^J \to W^J \sqcup \{0\}$ by
\[
\te_i w = \begin{cases} s_i w & \text{if } \ell(s_i w) < \ell(w) \\ 0 & \text{otherwise}, \end{cases}
\qquad\quad 
\tf_i w = \begin{cases} s_i w & \text{if } \ell(s_i w) > \ell(w) \text{ and } s_i w \in W^J \\ 0 & \text{otherwise}, \end{cases}
\]
and weight function $\wt(w) = \fw_r - \alpha_{i_1} - \cdots - \alpha_{i_{\ell}}$, where $s_{i_1} \cdots s_{i_{\ell}}$ is any reduced expression for $w \in W^J$. By Stembridge~\cite{Ste01}, this is well-defined, and this gives $W^J$ the structure of a crystal associated to the minuscule representation $V(\fw_r)$~\cite{Scr20}.

Now we give an explicit construction of the minuscule representation.\footnote{A different construction was also recently given~\cite{DDW21}, which appeared while writing this appendix.}

\begin{prop}
Consider the vector space
\[
\mathbb{V}(\fw_r) := \Span_{\field} \{ v_w \mid w \in W^J \}.
\]
Then $\mathbb{V}(\fw_r)$ is made into a $\g$-representation by
\[
e_i v_w = v_{\te_i w},
\qquad\qquad
f_i v_w = v_{\tf_i w},
\qquad\qquad
h_i v_w = \inner{\wt(w)}{h_i} v_w,
\qquad\qquad
\]
where $v_0 = 0$, and extended by linearity. Furthermore $\mathbb{V}(\fw_r) \iso V(\fw_r)$ as $\g$-representations.
\end{prop}

\subsection{Implementation}

We now give our implementation using \textsc{SageMath}.
For our crystals, we will use the realization using rigged configurations~\cite{Sch06,SS21}.

We build the minuscule representation $V(\fw_1)$ in type $E_6$, which is constructed as $\mathbb{V}(\fw_1)$:

\begin{lstlisting}
sage: La=RootSystem(['E',6]).weight_lattice().fundamental_weights()
sage: M = crystals.RiggedConfigurations(La[1])
sage: VM = ReprMinuscule(M, QQ)
sage: v = VM.maximal_vector()
\end{lstlisting}

Let $v$ denote the highest weight vector of $V(\fw_1)$. There exists a highest weight vector $v_{\fw_4}$ of weight $\fw_4$ in $V(\fw_1)^{\otimes 3}$. Explicitly, it is given as
\begin{align*}
v_{\fw_4} & = v \otimes f_1 v \otimes f_3 f_1 v - f_1 v \otimes v \otimes f_3 f_1 v - v \otimes f_3 f_1 v \otimes f_1 v 
\\ & \hspace{20pt} + f_1 v \otimes f_3 f_1 v \otimes v + f_3 f_1 v \otimes v \otimes f_1 v - f_3 f_1 v \otimes f_1 v \otimes v.
\end{align*}
It is a finite computation to show this is a highest weight vector.
We also perform this computation in \textsc{SageMath}:

\begin{lstlisting}
sage: x = (tensor([v, v.f(1), v.f(1).f(3)]) 
....:      - tensor([v.f(1), v, v.f(1).f(3)])
....:      - tensor([v, v.f(1).f(3), v.f(1)])
....:      + tensor([v.f(1), v.f(1).f(3), v])
....:      + tensor([v.f(1).f(3), v, v.f(1)])
....:      - tensor([v.f(1).f(3), v.f(1), v])
....:      )
sage: all(x.e(i) == 0 for i in M.index_set())
True
\end{lstlisting}

Thus, we can build a $\g$-representation by $\VV(\fw_4)  := \langle v_{\fw_4} \rangle \subseteq V(\fw_1)^{\otimes 3}$, and since the decomposition of tensor products of finite dimensional $\g$-representations are determined by computing highest weight vectors, we have $\mathbb{V}(\fw_4) \iso V(\fw_4)$.
In order to do computations, we need to construct a weight basis for $\VV(\fw_4)$.
We do so by using the crystal $B(\fw_4)$.
Let $b \in B(\fw_4)$, then define $v_b = f_{i_1} \cdots f_{i_{\ell}} v_{\fw_4}$, where $b = \tf_{i_1} \cdots \tf_{i_{\ell}} u_{\fw_4}$ for some fixed path $(i_1, \dotsc, i_{\ell})$.
Define $\mcB := \{v_b \mid b \in B(\fw_4)\}$.
Clearly this may depend on the choice of path from $u_{\fw_4} \to b$, but irregardless of this choice, we have $v_b \in V(\fw_4)_{\wt(b)}$.

\begin{figure}
\[
\newcommand{\rt}[2]{\begin{array}{c@{}c@{}c@{}c@{}c} &&#2\\ #1 \end{array}}
\begin{tikzpicture}[scale=2,>=latex]
\node (a1) at (0,0) {$\rt{1&0&0&0&0}{0}$};
\node (a2) at (3,0) {$\rt{0&0&0&0&0}{1}$};
\node (a3) at (1,0) {$\rt{0&1&0&0&0}{0}$};
\node (a4) at (2,0) {$\rt{0&0&1&0&0}{0}$};
\node (a5) at (4,0) {$\rt{0&0&0&1&0}{0}$};
\node (a6) at (5,0) {$\rt{0&0&0&0&1}{0}$};
\node (a13) at (0,1) {$\rt{1&1&0&0&0}{0}$};
\node (a34) at (1,1) {$\rt{0&1&1&0&0}{0}$};
\node (a24) at (2,1) {$\rt{0&0&1&0&0}{1}$};
\node (a45) at (3,1) {$\rt{0&0&1&1&0}{0}$};
\node (a56) at (4,1) {$\rt{0&0&0&1&1}{0}$};
\node (a134) at (0,2) {$\rt{1&1&1&0&0}{0}$};
\node (a234) at (1,2) {$\rt{0&1&1&0&0}{1}$};
\node (a345) at (2,2) {$\rt{0&1&1&1&0}{0}$};
\node (a245) at (3,2) {$\rt{0&0&1&1&0}{1}$};
\node (a456) at (4,2) {$\rt{0&0&1&1&1}{0}$};
\node (a1234) at (0,3) {$\rt{1&1&1&0&0}{1}$};
\node (a1345) at (1,3) {$\rt{1&1&1&1&0}{0}$};
\node (a2345) at (2,3) {$\rt{0&1&1&1&0}{1}$};
\node (a3456) at (3,3) {$\rt{0&1&1&1&1}{0}$};
\node (a2456) at (4,3) {$\rt{0&0&1&1&1}{1}$};
\node (a12345) at (1,4) {$\rt{1&1&1&1&0}{1}$};
\node (a13456) at (2,4) {$\rt{1&1&1&1&1}{0}$};
\node (a23445) at (3,4) {$\rt{0&1&2&1&0}{1}$};
\node (a23456) at (4,4) {$\rt{0&1&1&1&1}{1}$};
\node (a123445) at (2,5) {$\rt{1&1&2&1&0}{1}$};
\node (a123456) at (3,5) {$\rt{1&1&1&1&1}{1}$};
\node (a234456) at (4,5) {$\rt{0&1&2&1&1}{1}$};
\node (a1233445) at (2,6) {$\rt{1&2&2&1&0}{1}$};
\node (a1234456) at (3,6) {$\rt{1&1&2&1&1}{1}$};
\node (a2344556) at (4,6) {$\rt{0&1&2&2&1}{1}$};
\node (a12334456) at (2,7) {$\rt{1&2&2&1&1}{1}$};
\node (a12344556) at (4,7) {$\rt{1&1&2&2&1}{1}$};
\node (a123344556) at (3,8) {$\rt{1&2&2&2&1}{1}$};
\node (a1233444556) at (3,9) {$\rt{1&2&3&2&1}{1}$};
\node (a12233444556) at (3,10) {$\rt{1&2&3&2&1}{2}$};
\draw[->,dgreencolor] (a1) -- node[midway,right] {\tiny $3$} (a13);
\draw[->,red] (a3) -- node[midway,right] {\tiny $1$} (a13);
\draw[->,UQpurple] (a3) -- node[midway, right] {\tiny $4$} (a34);
\draw[->,UQpurple] (a2) -- node[midway,right] {\tiny $4$} (a24);
\draw[->,blue] (a4) -- node[midway,right] {\tiny $2$} (a24);
\draw[->,dgreencolor] (a4) -- node[midway,right] {\tiny $3$} (a34);
\draw[->,UQgold] (a4) -- node[midway,left] {\tiny $5$} (a45);
\draw[->,UQpurple] (a5) -- node[midway,right] {\tiny $4$} (a45);
\draw[->,black] (a5) -- node[midway,left] {\tiny $6$} (a56);
\draw[->,UQgold] (a6) -- node[midway,right] {\tiny $5$} (a56);
\draw[->,UQpurple] (a13) -- node[midway,right] {\tiny $4$} (a134);
\draw[->,red] (a34) -- node[midway,right] {\tiny $1$} (a134);
\draw[->,blue] (a34) -- node[midway,right] {\tiny $2$} (a234);
\draw[->,UQgold] (a34) -- node[midway,left] {\tiny $5$} (a345);
\draw[->,dgreencolor] (a24) -- node[midway,right] {\tiny $3$} (a234);
\draw[->,UQgold] (a24) -- node[midway,left] {\tiny $5$} (a245);
\draw[->,dgreencolor] (a45) -- node[midway,right] {\tiny $3$} (a345);
\draw[->,UQpurple] (a56) -- node[midway,right] {\tiny $4$} (a456);
\draw[->,black] (a45) -- node[midway,left] {\tiny $6$} (a456);
\draw[->,blue] (a45) -- node[midway,right] {\tiny $2$} (a245);
\draw[->,blue] (a134) -- node[midway,right] {\tiny $2$} (a1234);
\draw[->,UQgold] (a134) -- node[midway,left] {\tiny $5$} (a1345);
\draw[->,red] (a234) -- node[midway,right] {\tiny $1$} (a1234);
\draw[->,UQgold] (a234) -- node[midway,left] {\tiny $5$} (a2345);
\draw[->,red] (a345) -- node[midway,right] {\tiny $1$} (a1345);
\draw[->,blue] (a345) -- node[midway,right] {\tiny $2$} (a2345);
\draw[->,black] (a345) -- node[midway,left] {\tiny $6$} (a3456);
\draw[->,dgreencolor] (a245) -- node[midway,right] {\tiny $3$} (a2345);
\draw[->,black] (a245) -- node[midway,left] {\tiny $6$} (a2456);
\draw[->,dgreencolor] (a456) -- node[midway,right] {\tiny $3$} (a3456);
\draw[->,blue] (a456) -- node[midway,right] {\tiny $2$} (a2456);
\draw[->,UQgold] (a1234) -- node[midway,left] {\tiny $5$} (a12345);
\draw[->,blue] (a1345) -- node[midway,right] {\tiny $2$} (a12345);
\draw[->,black] (a1345) -- node[midway,left] {\tiny $6$} (a13456);
\draw[->,black] (a2345) -- node[midway,left] {\tiny $6$} (a23456);
\draw[->,red] (a2345) -- node[midway,right] {\tiny $1$} (a12345);
\draw[->,UQpurple] (a2345) -- node[midway,left] {\tiny $4$} (a23445);
\draw[->,red] (a3456) -- node[midway,right] {\tiny $1$} (a13456);
\draw[->,blue] (a3456) -- node[midway,left] {\tiny $2$} (a23456);
\draw[->,dgreencolor] (a2456) -- node[midway,right] {\tiny $3$} (a23456);
\draw[->,UQpurple] (a12345) -- node[midway,left] {\tiny $4$} (a123445);
\draw[->,black] (a12345) -- node[midway,left] {\tiny $6$} (a123456);
\draw[->,blue] (a13456) -- node[midway,left] {\tiny $2$} (a123456);
\draw[->,red] (a23445) -- node[midway,right] {\tiny $1$} (a123445);
\draw[->,black] (a23445) -- node[midway,left] {\tiny $6$} (a234456);
\draw[->,red] (a23456) -- node[midway,right] {\tiny $1$} (a123456);
\draw[->,UQpurple] (a23456) -- node[midway,right] {\tiny $4$} (a234456);
\draw[->,dgreencolor] (a123445) -- node[midway,right] {\tiny $3$} (a1233445);
\draw[->,black] (a123445) -- node[midway,left] {\tiny $6$} (a1234456);
\draw[->,UQpurple] (a123456) -- node[midway,right] {\tiny $4$} (a1234456);
\draw[->,UQgold] (a234456) -- node[midway,right] {\tiny $5$} (a2344556);
\draw[->,red] (a234456) -- node[midway,right] {\tiny $1$} (a1234456);
\draw[->,red] (a2344556) -- node[midway,right] {\tiny $1$} (a12344556);
\draw[->,black] (a1233445) -- node[midway,right] {\tiny $6$} (a12334456);
\draw[->,dgreencolor] (a1234456) -- node[midway,right] {\tiny $3$} (a12334456);
\draw[->,UQgold] (a1234456) -- node[midway,left] {\tiny $5$} (a12344556);
\draw[->,UQgold] (a12334456) -- node[midway,above left] {\tiny $5$} (a123344556);
\draw[->,dgreencolor] (a12344556) -- node[midway,above right] {\tiny $3$} (a123344556);
\draw[->,UQpurple] (a123344556) -- node[midway,right] {\tiny $4$} (a1233444556);
\draw[->,blue] (a1233444556) -- node[midway,right] {\tiny $2$} (a12233444556);
\end{tikzpicture}
\]
\caption{The root poset $\Phi^+$ in type $E_6$.}
\label{fig:root_poset}
\end{figure}
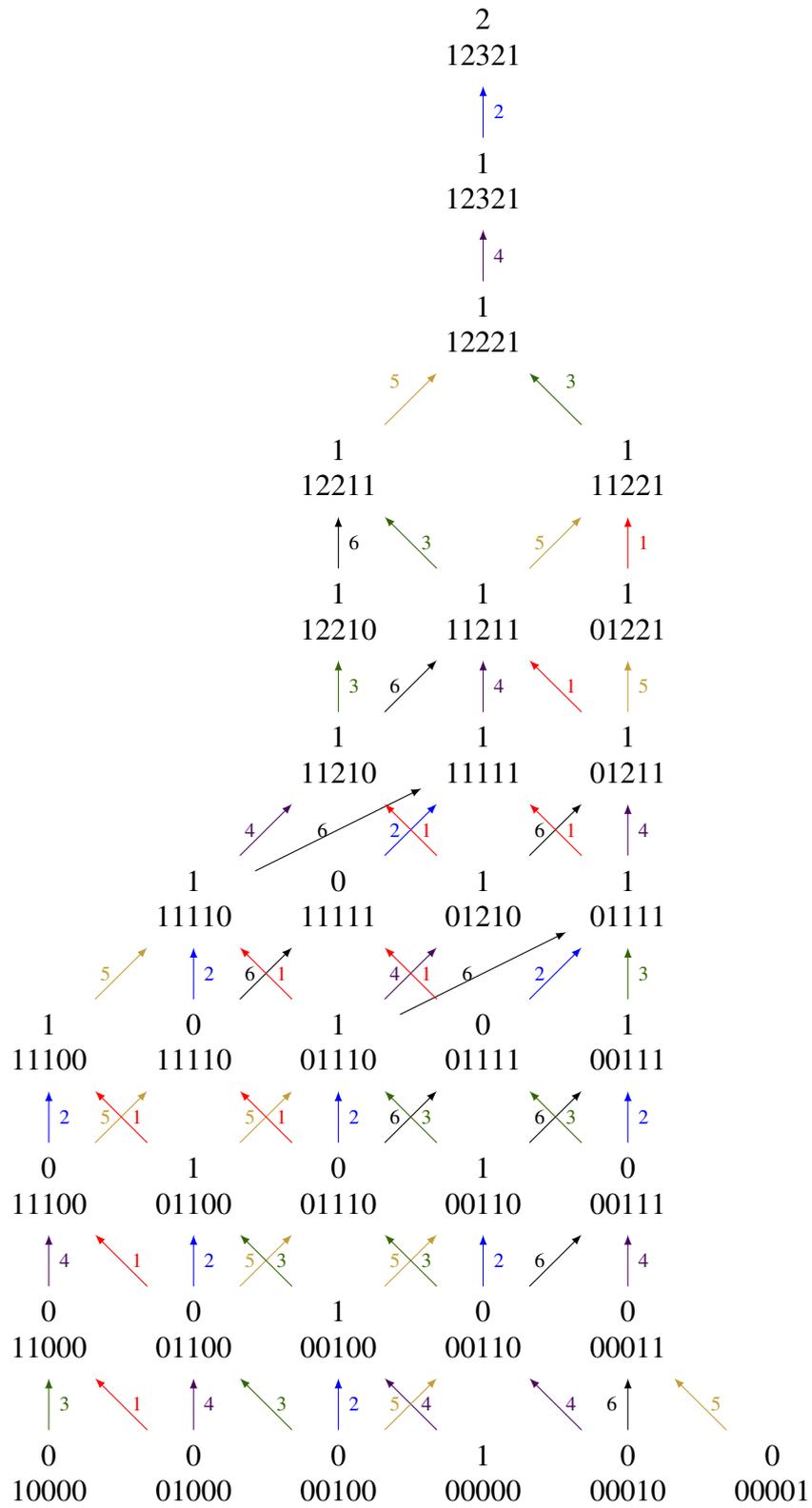

Below, we construct $\VV(\fw_4)$ in \textsc{SageMath} by using $\mcB$ as follows.
For each $b \in B(\fw_4)$, we take the path recursively constructed by taking the minimal $i_k$ such that we have a path to $b$ from $f_{i_k} \cdots f_{i_{\ell}} u _{\fw_4}$ (although any such path could do).
This gives us a set of elements $\mcB$, and we need to show that $\mcB$ are linearly independent.
We verify this by seeing the rank of the matrix of these vectors is $2925 = \dim V(\fw_4) = \abs{B(\fw_4)}$.
Furthermore, we verify that this does give us a $\g$-representation by checking all of the relations are satisfied on each basis element:

\begin{lstlisting}
sage: S=SubRepresentation(x,crystals.RiggedConfigurations(La[4]))
sage: verify_representation(S)  # long time -- few minutes 
\end{lstlisting}

Next, consider the positive roots
\begin{align*}
\beta& = \alpha_1 + \alpha_2 + 2 \alpha_3 + 3 \alpha_4 + 2 \alpha_5 + \alpha_6 = \begin{array}{c@{}c@{}c@{}c@{}c} &&1\\ 1&2&3&2&1 \end{array},
\\
\theta & = \alpha_1 + 2\alpha_2 + 2 \alpha_3 + 3 \alpha_4 + 2 \alpha_5 + \alpha_6 = \begin{array}{c@{}c@{}c@{}c@{}c} &&2\\ 1&2&3&2&1 \end{array}.
\end{align*}
Since all of the root spaces in $\g$ are 1 dimensional (that is $\dim \g_{\alpha} = 1$ for all $\alpha \in \Phi$), we construct the basis element $f_{\gamma}$ of $\g_{\gamma}$ (which forms the unique basis up to scalar) by finding some sequence $(i_1, i_2, \dotsc, i_{\ell})$ such that
\[
\sum_{j=1}^k \alpha_{i_j} \in \Phi^+,
\qquad\qquad
\sum_{j=1}^{\ell} \alpha_{i_j} = \gamma,
\]
for all $1 \leq k \leq \ell$.
In particular, we can take any path from $\alpha_{i_1}$ to $\gamma$ in Figure~\ref{fig:root_poset}.
Then we have
\[
f_{\gamma} = [\cdots[[f_{i_1}, f_{i_2}], f_{i_3}] \cdots f_{i_{\ell}}].
\]
We write $f_{\gamma}$ in the free algebra generated by $\langle f_i \rangle_{i \in I}$ using the commutator property $[X, Y] = XY - YX$ and apply the result to any vector in the $\g$-representation.
In other words, we compute
\[
f_{\gamma} = \sum_{\ba} \pm f_{a_1} \cdots f_{a_{\ell-1}} f_{a_{\ell}},
\qquad\qquad
f_{\gamma} v = \sum_{\ba} \pm (f_{a_1} \cdots (f_{a_{\ell-1}} (f_{a_{\ell}} v)) \cdots).
\]
Using this process, we construct the vector $v = f_{\theta} f_{\beta} v_{\fw_4}$:

\begin{lstlisting}
sage: v = S.maximal_vector()
sage: al = RootSystem(['E',6]).root_lattice().simple_roots()
sage: b1 = al[1] + 1*al[2] + 2*al[3] + 3*al[4] + 2*al[5] + al[6]
sage: b2 = al[1] + 2*al[2] + 2*al[3] + 3*al[4] + 2*al[5] + al[6]
sage: ops = build_root_operators(['E',6])
sage: vzero = apply_f_operators(ops[b2], 
....:              apply_f_operators(ops[b1], v))
sage: vzero != 0
True
\end{lstlisting}

Lastly, we construct the orbit up to sign and show that it spans a $45$ dimensional vector space:

\begin{lstlisting}
sage: orbit = set([vzero])
sage: nl = [vzero]
sage: I = CartanType(['E',6]).index_set()
sage: while nl:
....:     cur = nl
....:     nl = []
....:     for vec in cur:
....:         for i in I:
....:             vs = vec.s(i)
....:             if vs in orbit or -vs in orbit:
....:                 continue
....:             orbit.add(vs)
....:             nl.append(vs)
sage: len(orbit)
240
sage: wt0 = [b for b in S.basis().keys() if b.weight() == 0]
sage: matrix([[vec[b] for b in wt0] for vec in orbit]).rank()
45
\end{lstlisting}

\begin{remark}
The \textsc{SageMath} code for the implementation in this appendix is included as an ancillary file on \arxiv{2010.11357}.
\end{remark}





\end{document}